\documentclass[12pt]{article}
\usepackage[centertags]{amsmath}
\usepackage{amsfonts}
\usepackage{amssymb}
\usepackage{amsthm}
\usepackage{amscd}
\usepackage{color}
\usepackage[textsize=small]{todonotes}
\usepackage{cases}
\usepackage{verbatim}
\usepackage[mathscr]{euscript}
\usepackage{graphicx}
\usepackage{bbm}
\usepackage{epsfig}
\usepackage{fancyhdr}
\usepackage{hyperref}
\usepackage{wrapfig}
\usepackage[small,nohug,heads=vee]{diagrams}
\diagramstyle[labelstyle=\scriptstyle]

\theoremstyle{plain}
\newtheorem{thm}{Theorem}[section]
\newtheorem{cor}[thm]{Corollary}
\newtheorem{lem}[thm]{Lemma}
\newtheorem{prop}[thm]{Proposition}
\theoremstyle{definition}

\newtheorem{rem}[thm]{Remark}
\newtheorem{ex}[thm]{Example}

\newcounter{thestep}
\newenvironment{step}
{\medskip\refstepcounter{thestep}{{\textit{Step~\arabic{thestep}.}}}\begin{itshape}}
{\end{itshape}\medskip }

\newcounter{thehypothesis}
\newenvironment{hypothesis}
{\refstepcounter{thehypothesis}\begin{itemize}\item[(H\arabic{thehypothesis})]}
{\end{itemize}}

\newcounter{theproperty}
\newenvironment{property}
{\refstepcounter{theproperty}\begin{itemize}\item[(P\arabic{theproperty})]}
{\end{itemize}}

\newcommand{\compcent}[1]{\vcenter{\hbox{$#1\circ$}}}
\newcommand{\comp}{\mathbin{\mathchoice {\compcent\scriptstyle}{\compcent\scriptstyle}
{\compcent\scriptscriptstyle}{\compcent\scriptscriptstyle}}}

\renewcommand{\l}{\lambda}
\newcommand{\om}{\omega}
\newcommand{\Om}{\Omega}

\renewcommand{\a}{\alpha}
\renewcommand{\b}{\beta}
\newcommand{\e}{\varepsilon}

\newcommand{\duzhky}[1]{{\left ( #1 \right )}}
\newcommand{\im}[1]{{\mathrm{Im}\,  #1  }}
\newcommand{\re}[1]{{\mathrm{Re}\,  #1  }}

\newcommand{\grad}{{\mathrm{grad}\,}}
\newcommand{\sympred}{{/\!\! /}}
\newcommand{\cd}[1]{\nabla_{\! #1}\,}
\newcommand{\CS}{{\operatorname{\mathrm{CS}}}}
\newcommand{\C}{{\mathbb{C}}}
\newcommand{\R}{{\mathbb{R}}}

\DeclareMathOperator{\sign}{sign}
\DeclareMathOperator{\coker}{coker}
\DeclareMathOperator{\codim}{codim}

\DeclareMathOperator{\Ind}{Ind}

\newcommand{\ssst}[1]{{\scriptscriptstyle #1}}

\begin{document}

\title{Fukaya--Seidel category and gauge theory}%
\author{Andriy Haydys\\%
\textit{University of Bielefeld}\\%
}%

\date{September 21, 2012}


\maketitle

\begin{abstract}
A new construction of the Fukaya--Seidel category associated with a symplectic Lefschetz fibration is outlined. Applying this construction in an infinite dimensional case, a Fukaya--Seidel-type category is associated with a smooth three-manifold. In this case the construction is based on a five-dimensional gauge theory.
\end{abstract}

\section{Introduction}

This paper consists of two major parts. In the first part, based on the idea of Seidel~\cite{Seidel:01VanishingCyclesI} we outline a construction of the Fukaya--Seidel category, which is associated with a symplectic manifold $M$ equipped with the structure of a symplectic Lefschetz fibration. By this we mean, roughly speaking, a choice of an almost complex structure  $J$ and a $J$-holomorphic Morse function $f$.  This construction does not rely on the notion of vanishing cycle but emphasizes instead the role of the antigradient flow lines of $\re (e^{i\theta} f)$. In the second part, this construction is applied in the infinite dimensional case of the complex Chern--Simons functional. The corresponding construction conjecturally associates a Fukaya--Seidel--type category to a smooth three-manifold.

Our motivation originated from the suggestion to use higher dimensional gauge theory in studies of low dimensional manifolds as outlined in~\cite{Haydys:09_GaugeTheory_jlms}. Namely, suppose we are given a construction that associates a  higher dimensional manifold $W_X$  to each lower dimensional manifold $X$ from a suitable subclass and possibly equipped with an additional structure. The manifold $W_X$ is assumed to be of dimension $6,7$ or $8$ and endowed with an $SU(3),  G_2$ or $Spin(7)$ structure, respectively. Then, by counting  higher dimensional instantons on $W_X$ we should obtain an invariant of $X$. The construction studied in~\cite{Haydys:09_GaugeTheory_jlms} in detail associates to each smooth spin four-manifold the total space of its spinor bundle. 

Another construction of a similar nature associates to $X^4$ the total space of the ``twisted spinor bundle'' $\underline{\mathbb R}\oplus \Lambda^2_+T^*X$. Then $Spin(7)$-instantons invariant along each fibre are solutions of the Vafa--Witten equations~\cite{VafaWitten:94_StrongCouplingTest}, while $Spin(7)$-instantons invariant only along the fibres of $\Lambda^2_+T^*X$ can be interpreted as antigradient flow lines of a function, whose critical points are solutions of the Vafa--Witten equations. It turns out that these flow lines can be obtained from certain elliptic equations on a general five-manifold $W^5$ equipped with a nonvanishing vector field by specializing to the case $W=X^4\times\mathbb R$ just like flow lines of the real Chern--Simons functional are obtained from the anti-self-duality equations on $X^4=Y^3\times\mathbb R$.  Specializing further to $W=Y^3\times\mathbb R^2$ we obtain a construction of a Fukaya-type $A_\infty$-category (this requires some extra choices) just like specialization of the anti-self-duality equations to $\Sigma^2\times\mathbb R^2$ leads to the construction of the Fukaya $A_\infty$-category associated with $\Sigma$. At this point an important distinction from the case of Riemann surfaces emerges. Namely, the construction involves a natural holomorphic function, the complex Chern--Simons functional, and this has significant implications for the flavour of the construction. 

Having said this though, we do not appeal in this paper to higher dimensional anti-self-duality equations but rather begin directly with the formulation of the five-dimensional gauge theory. From this perspective the most interesting theory is obtained via reduction to three-manifolds, where the construction of the $A_\infty$-category  admits a finite-dimensional interpretation in the framework of symplectic geometry.

\medskip

This paper is organized as follows. In Section~\ref{Sect_FukSeidelCateg} we describe the construction of the Fukaya--Seidel category in the finite-dimensional case. From one point of view, this construction is a generalization of a Floer theory, where \emph{generators} of the homology groups are antigradient flow lines of the real part of a holomorphic Morse function connecting a pair of critical points. Then the Floer differential is obtained from pseudoholomorphic planes with a Hamiltonian perturbation satisfying certain asymptotic conditions (see~\eqref{Eq_PseudoholomPlaneWithHamPert_mod}--~\eqref{Eq_PseudoholomPlaneWithHamPert_BC_s} for more details).

Sections~\ref{Sect_5dGaugeTheory} and~\ref{Sect_DimReductions} are devoted to the formulation of the five-dimensional gauge theory and its various dimensional reductions. In Section~\ref{Sect_Invariants} we describe applications of the equations obtained in the previous sections to low dimensional topology. In particular, one can (conjecturally) associate an integer to a five-manifold, Floer-type homology groups to a four-manifold and a Fukaya--Seidel-type category to a three-manifold. In dimension three, critical points correspond to flat $G^c$-connections on $Y$, flow lines correspond to Vafa--Witten-type instantons on $Y\times\mathbb R$ and pseudoholomorphic planes correspond to ``five-dimensional instantons" on $Y\times\mathbb R^2$. This should be a part of a multi-tier (extended) quantum field theory~\cite{Freed:09_RemarksOnChernSimons} but we do not study this aspect in the current paper.

The constructions described in this paper may also be useful in other settings, for instance in the context of Calabi--Yau threefolds. Here the critical points of the holomorphic Chern--Simons functional correspond to holomorphic vector bundles over a Calabi--Yau threefold $Z$, flow lines correspond to $G_2$-instantons on $Z\times\mathbb R$ and pseudoholomorphic planes correspond to $Spin(7)$-instantons on $Z\times\mathbb R^2$. 

\medskip

Many aspects of this paper are related to ideas of various authors. As it has been already mentioned above, our construction of the Fukaya--Seidel category in the finite dimensional case is a modification of Seidel's idea. The equation we utilize for the definition of the structure maps in the Fukaya--Seidel $A_\infty$-category was used in the context of mirror symmetry by Fan--Jarvis--Ruan~\cite{FanJarvisRuan:07_WittenEqn} (``Witten equation'') in the case of quasi-homogeneous polynomials. The antigradient flow lines of the real part of the holomorphic Chern--Simons functional appeared in~\cite{KapustinWitten:07_ElectrMagnDuality} for the first time and were further studied by Witten~\cite{Witten:10_AnalyticContinuation,Witten:11_NewLookAtPathIntegral}. Donaldson and Segal~\cite{DonaldsonSegal:09} used antigradient flow lines of the real part of the holomorphic Chern--Simons functional in the context of Calabi--Yau threefolds.  

\medskip

\textsc{Acknowledgements:} I thank S.Bauer, S.Donaldson, V.Pidstrygach, V.Rabinovich, Y.Ruan, D.Salamon and P.Seidel for helpful discussions and also B.Himpel for reading the draft of this paper. I acknowledge the financial support of the German Research Foundation (DFG) and the hospitality of Imperial College London, where part of this work was carried out.


\section{Fukaya--Seidel categories of symplectic Lef\-schetz fibrations}\label{Sect_FukSeidelCateg}

In~\cite{Seidel:01VanishingCyclesI, Seidel:08FukayaCategories} Seidel describes the construction of a Fukaya category associated with a symplectic Lefschetz fibration  in terms of vanishing cycles. In the first part of this section we describe omitting (important) technical details  an alternative approach, which does not rely on the notion of vanishing cycle. The rest of the section is devoted to basic analytic properties of the objects involved in the construction.

\subsection{Symplectic Lefschetz fibrations}

Let $(M^{2n},\om,\l),\   \om=d\l$, be an exact symplectic manifold with boundary. Choose an almost complex structure $J$ such that $g(\cdot,\cdot)=\om(\cdot, J\cdot)$ is a Riemannian metric on $M$. It is also convenient to assume that $J$ is orthogonal with respect to $g$. Let $f\colon M\rightarrow\mathbb C$ be a  $J$-holomorphic function. 
We assume the following properties:
\begin{property}\label{Pty_Properness}
  $f$ is a proper map with finitely many non--degenerate critical points lying in pairwise different fibres. Moreover,  locally near each critical point $J$ is integrable.
\end{property}
\begin{property}\label{Pty_ConvexBoundary}
The boundary of $M$ is weakly $J$--convex.
\end{property}
\begin{property}\label{Pty_LocTrivAwayFromK}
  Let $M_0=f^{-1}(z_0)$ be a regular fiber. Then there exist compact subsets $K\subset M\setminus\partial M$, $K'\subset M_0\times\C\setminus\partial M_0\times\C$, and a positive number $r$ with the following significance. Denote $V=M\setminus K,\ V'=M_0\times\C\setminus K'$. Then for each $z\in\C$ there exists a small neighbourhood $B_\delta(z)$ and a fiber preserving diffeomorphism  $\psi_z$ such that the following holds: The diagram
\begin{diagram}
V\cap \bigl( M_0\times B_\delta(z)\bigr ) &  & \rTo^{\quad\psi_z\quad} & &V'\cap f^{-1}\bigl (B_\delta(z)\bigr )\\
   &\rdTo^{pr_2}      &  &\ldTo^{f} & \\
    &             &   B_\delta (z)   &
\end{diagram}
commutes, $\psi_z$ is the identity on $(M_0\times\{z_0 \})\cap V$ whenever $z_0\in B_\delta(z)$, and  the pull--back of $(\l, J)$ is $(\l_{M_0}+r\l_{0}, J_{M_z}\times I_{0})$. Here $\l_{0}=\re{(izd\bar z)}$ is the primitive of the standard symplectic form $\om_0$ and $I_{0}$ is the standard complex structure on $\C$.
\end{property}


It is worth pointing out that properties~(P\ref{Pty_Properness}) and~(P\ref{Pty_LocTrivAwayFromK}) imply that there exists $R>0$ such that the preimage of $B_R^c(0)=\{  |z|> R\}$ is contained in $V$. In other words, for any $z\in B_R^c(0)$ there exists a 
neighbourhood  $B_\delta(z)\subset B_R^c(0)$, and fiber preserving diffeomorphism  $\psi_z\colon M_0\times B_\delta(z)\to f^{-1}(B_\delta(z))$ with the properties as in~(P\ref{Pty_LocTrivAwayFromK}). Similarly, there exists a neighbourhood $W$ of $\partial M$, a neighbourhood $W'$ of $\partial (M_0\times\C)=\partial M_0\times\C$, and a diffeomorphism $\psi\colon W'\to W$ such that the pull-back of  $(\l, J)$ is $(\l_{M_0}+r\l_{0}, J_{M_0}\times I_{0})$. Here $M_0$ is some fiber. Conversely, these two properties imply~(P\ref{Pty_LocTrivAwayFromK}).

Denote
\[
f= f_0+i f_1,\qquad\rho= \{f_0, f_1 \}.
\]
An easy computation shows that on $V$ we have $\rho=r^{-1}$.  In particular, this implies that $\rho$ is bounded on $M$.

The following interpretation of $\rho$ will be useful in the sequel. Denote
\begin{equation*}
    v_0=\grad  f_0 \quad\text{and}\quad v_1=\grad f_1.
\end{equation*}
The holomorphicity of $f$  implies $Jv_0=v_1$. Then the Hamiltonian vector field of $f_0$ is $X_{f_0}=-Jv_0=-v_1$. This yields
\begin{equation}\label{Eq_FuncRho}
\rho (m)=| v_0(m) |^2=|v_1(m)|^2.
\end{equation}

\begin{rem}
  It is interesting to notice that property~(P\ref{Pty_LocTrivAwayFromK}) is in fact equivalent to $\rho$ being constant on a complement of a compact subset. Indeed, assume $\rho$ is constant on $V=M\setminus K$. Then the identity $[v_0, v_1]=-[X_{f_1}, X_{f_0}]=X_{\{f_0,f_1\}}$
implies that $v_0$ and $v_1$ commute on $V$. The subset $M\setminus \mathrm{Crit}(f)$  is equipped with the connection, which is induced by the symplectic form. Then   $\rho^{-1}v_0$ and $\rho^{-1}v_1$ are the horizontal lifts of $\tfrac \partial{\partial s}$ and $\tfrac\partial{\partial t}$,  respectively, where  $(s,t)$ be coordinates on $\C\cong \R^2$. Hence, the connection is flat over $V$. It follows that in a flat trivialization in a neighbourhood of some $z\in \C$ the symplectic form can be written as $\om_{M_z}+r\om_0$, where $r$ is some function. Then $r$ is constant, since $r^{-1}=\rho$.
\end{rem}

Examples of the fibrations with properties (P\ref{Pty_Properness})--(P\ref{Pty_LocTrivAwayFromK}) can be found in~\cite[(19b)]{Seidel:08FukayaCategories} (it is only needed to drop the restriction to the preimage of a large disc).

\medskip

Other examples can be constructed starting from symplectic Lefschetz fibrations over the disc $\pi\colon E\to D=B_1(0)$ as in~\cite[(15a)]{Seidel:08FukayaCategories} assuming triviality near the horizontal boundary~\cite[Remark~15.2]{Seidel:08FukayaCategories}. Indeed, first of all on an open neighbourhood of $\partial^hE$ diffeomorphic to an open neighbourhood of the horizontal boundary of the trivial fibration $E_{pt}\times D$ with the help of a suitable cut-off function we can deform  the symplectic form to $\om_{E_{pt}}$. This is clearly no longer symplectic form on the horizontal subbundle but later on we will add some multiple of the standard symplectic form on $D$ so that the resulting 2--form will be symplectic on the total space. 

To extend $E$ to a fibration over the whole complex plane proceed as follows. Choose $\delta>0$ such that all critical values of $\pi$ are contained in $B_{1-\delta}(0)$. With the help of the parallel transport along radial lines we obtain 
\begin{equation}\label{Eq_AuxIsomorphismOfLefschetzFibr}
 \bigl. E\bigr |_{Z}\cong pr^*\bigl. E\bigr |_{S^1_{1-\delta}},
\end{equation}
where $pr\colon Z=\{1-\delta\le |z|\le 1\}\cong S^1_{1-\delta}\times [1-\delta, 1]\to  [1-\delta, 1]$. If $(\varrho,\varphi)$ denote the polar coordinates we can write~\cite[(15a)]{Seidel:08FukayaCategories} the symplectic 2--form on $pr^*\bigl. E\bigr |_{S^1_{1-\delta}}$ in the form 
\[
\om = \om_{E_{1-\delta,0}}+ d\kappa,
\]  
where $\kappa=\kappa_1(\varrho, \varphi) d\varrho + \kappa_2(\varrho, \varphi) d\varphi + dR$ for some functions $\kappa_1,\kappa_2, R\in C^\infty\bigl ([1-\delta, 1]\times S^1\times E_{1-\delta, 0}\bigr )$ (the notation does not reflect the dependence on all variables). Choose smooth cut--off functions $\alpha, \beta\colon [1-\delta,+\infty]\to [0,1]$ such that
\[
\alpha(\varrho)=
\begin{cases}
  \varrho\quad &\varrho\in [1-\delta, 1-\frac {2\delta}3],\\
 1   & \varrho\ge 1-\frac \delta 3,
\end{cases}
\qquad\quad
\beta(\varrho)=
\begin{cases}
  1\quad &\varrho\in [1-\delta, 1-\frac {2\delta}3],\\
  0  & \varrho\ge 1-\frac \delta 3,
\end{cases}
\]
and denote $\kappa' =\kappa_1(\alpha(\varrho),\varphi)\, d\varrho + \kappa_2(\alpha(\varrho), \varphi)\, d\varphi + d(\beta R)$. This defines a connection 1--form on 
\begin{equation}\label{Eq_AuxFibrationOverCylinder}
\bigl. E\bigr |_{S^1_{1-\delta}}\times [1-\delta,+\infty)\xrightarrow{\quad p\quad }
\bigl \{|z|>1-\delta \bigr \}.
\end{equation}
Then for sufficiently large $r>0$ the 2--form  $\om_r= \om_{E_{1-\delta,0}} +d\kappa' + rp^*\om_0$
is symplectic and equals to $\om_E+r\pi^*\om_0$ over $\{1-\delta <|z|<1-\frac{2\delta}3 \}$.  Hence, $\bigl. E\bigr|_{B_{1-\delta}(0)}$ can be glued with~\eqref{Eq_AuxFibrationOverCylinder} to obtain a fibration over the whole complex plane. By construction, this has properties (P\ref{Pty_Properness})--(P\ref{Pty_LocTrivAwayFromK}).



\subsection{Outline of the construction}\label{Subsect_OutlineOfConstruction}
The purpose of this subsection is to outline the main points of the alternative construction of the Fukaya--Seidel $A_\infty$-category. The discussion of technical details is postponed to the proceeding subsections.


\medskip

Let us briefly recall the basic ingredients of the Fukaya--Seidel $A_\infty$-category (see~\cite{ Seidel:01VanishingCyclesI,Seidel:08FukayaCategories} for details).  For the sake of simplicity we consider the ungraded version with coefficients in $\mathbb Z/2\mathbb Z$ (``preliminary version'' in the terminology of~\cite{Seidel:08FukayaCategories}). It is convenient to choose a basepoint $z_0$, which does not lie on any straight line determined by a pair of critical values (in particular, $z_0$ is distinct from critical values). Denote by $m_1,\dots, m_k$ critical points of $f$ and put $z_j=f(m_j)$. The indexing can be chosen such that the sequence $\arg (z_j-z_0)\in (-\pi, \pi]$ is decreasing 
in $j$ and this defines a linear order on the set of critical points.  

Choose a collection of paths connecting $z_0$ with each $z_j$ missing the remaining critical values. Let $L_j\subset f^{-1}(z_0)$ be the vanishing cycle of $m_j$ associated with the path connecting $z_0$ and  $z_j$.  Denote by $\Gamma$  the ordered collection $(L_1,\dots, L_k)$. Seidel associates to $\Gamma$ a directed Fukaya $A_\infty$-category $Lag^\to(\Gamma)$, whose objects are vanishing cycles $L_j$ and morphisms are Floer chain complexes  as follows. First recall that an $A_\infty$-structure is a collection of maps%
\begin{equation*}
 \mu^d\colon hom(L_{j_d}, L_{j_{d+1}})\otimes\dots\otimes hom(L_{j_1}, L_{j_{2}})\longrightarrow hom(L_{j_1}, L_{j_{d+1}}),\quad d=1,2,3,\dots
\end{equation*}
satisfying certain quadratic relations and by the directedness we have%
\begin{equation*}
  hom(L_j, L_k)=
   \begin{cases}
    CF(L_j, L_k) & j<k,\\
    \mathbb Z/2\cdot id &j=k,\\
    0 &j>k.
   \end{cases}
\end{equation*}
The Floer complex $CF(L_{j}, L_{k})$ is generated by the points of $L_j\cap L_k$ and the map $\mu^1$ is the Floer differential, which counts pseudoholomorphic strips such that one boundary component is mapped to $L_j$ and the other component is mapped to $L_k$. The maps $\mu^d$ for $d\ge 2$  are defined similarly by counting pseudoholomorphic discs with $d+1$ punctures on the boundary. The resulting $A_\infty$-category $Lag^\to(\Gamma)$ depends on the choices made but Seidel shows that the derived category $D^b(Lag^\to(\Gamma))$ is an invariant of the Lefschetz fibration.

\medskip

With this understood we now give another construction of the Fukaya--Seidel $A_\infty$-category. 
Pick a pair of critical points $(m_-,
m_+)$ and denote $\theta_\pm=\arg(z_\pm -z_0)\in (-\pi,\pi]$.  Let  $\gamma_m^\pm$ be the solution of the Cauchy problem
\begin{equation*}
 \dot\gamma_m^\pm +\cos\theta_\pm\, v_0 +\sin\theta_\pm\, v_1=0,\quad \gamma_m^\pm(0)=m\in f^{-1}(z_0).
\end{equation*}
Notice that the image of $f\comp\gamma_m^\pm\colon\mathbb R\rightarrow\mathbb C$ is contained in a straight line passing through $z_0$ and $z_\pm$. Then the vanishing cycle $L_\pm$ of $m_\pm$  associated with the segment $\overline{z_0z_\pm}$  can be
conveniently described as
\begin{equation*}
 L_\pm=\bigl\{ m\in f^{-1}(z_0)\; \bigl |\bigr .\ \;  \lim\limits_{t\to +\infty}\gamma_m^\pm(t)=m_\pm  \bigr\}.
\end{equation*}
Then, if we denote
\begin{equation}\label{Eq_FunctionTheta0}
\theta_0(t)=\begin{cases}
                      \theta_+ &\quad t\le 0,\\
                      \arg i(z_--z_0)=\theta_- \pm \pi &\quad  t>0,
                    \end{cases}
\end{equation}
the set  $L_+\cap L_-$ can be identified with the space of solutions of the problem%
\begin{equation}\label{Eq_BrokenFlowLine}
  \begin{aligned}
     &\ \;\dot \gamma +\cos\theta_0(t)\, v_0 +\sin\theta_0(t)\, v_1=0,\qquad
     \lim\limits_{t\to \pm\infty} \gamma(t) =m_\mp.
  \end{aligned}
\end{equation}
Here solutions are understood to be smooth on $\mathbb R\setminus \{0\}$ and continuous at $t=0$. We call solutions of~\eqref{Eq_BrokenFlowLine} \emph{broken flow lines} of $f$ connecting $m_-$ and $m_+$ and
denote by $\Gamma_0(m_-;m_+)$ 
the space of all solutions. Notice that for each broken antigradient flow line
$\gamma$ the image of $f\comp\gamma$ lies on the curve $\overline{z_-z_0z_+}$
and $f\comp\gamma(0)=z_0$.

It will be convenient in the sequel to replace $\theta_0$   by a smooth function $\theta_\nu$, where $\nu$ is a real parameter. The choice of the function $\theta_\nu$, which is described in Subsection~\ref{Subsect_AprioriC0Estim} in details,  turns out to be quite important, but what we need to know at this point is that $\theta_\nu$ is close to $\theta_0$ for $\nu$ small enough. 



Denote by $\Gamma_\nu=\Gamma_\nu (m_-, m_+)$ the space of solutions of the problem%
\begin{equation}\label{Eq_PertBrokenFlowLine}
     \dot \gamma +\cos\theta_\nu(t)\, v_0 +\sin\theta_\nu(t)\, v_1=0,\qquad 
     \lim\limits_{t\to \pm\infty} \gamma(t) =m_\mp.
\end{equation}
We  also call solutions of equations~\eqref{Eq_PertBrokenFlowLine} broken flow lines.

\begin{rem}
We assume that for $\nu$ small enough there exists a correspondence between solutions of~\eqref{Eq_PertBrokenFlowLine}  and~\eqref{Eq_BrokenFlowLine}. This is discussed in detail in Appendix~\ref{App_OnThePertAntigradFlowLines}.
\end{rem}


Further, notice that the Floer differential $\mu^1$ should take broken flow lines as
input and should return formal  linear combinations of broken flow lines as output. With $m_\pm$ as
above, pick additionally two solutions $\gamma_\pm$ of equations~\eqref{Eq_PertBrokenFlowLine}. Then the role
of holomorphic strips
with boundary on $L_\pm$ in our framework is played by solutions of the problem%
\begin{align}
    &\ \;\partial_s u + J\bigl (\partial_t u+\cos\theta_\nu(t)\, v_0 +\sin\theta_\nu(t)\, v_1\bigr )=0, & &\ \; u\colon \mathbb R^2_{s,t}\rightarrow M,\label{Eq_PseudoholomPlaneWithHamPert_mod}\\
    &\lim\limits_{t\to\pm\infty} u(s,t)= m_\mp,  & &\lim_{t\to\pm\infty}\int_{-\infty}^{+\infty} |\partial_s u(s,t)|\, ds=0,\label{Eq_PseudoholomPlaneWithHamPert_BC_t}\\
&  \lim\limits_{s\to\pm\infty} u(s,t)=\gamma_\mp(t),
   &\qquad  &\lim_{s\to\pm\infty}\int_{a}^b |\partial_su(s,t)|\, dt=0.\label{Eq_PseudoholomPlaneWithHamPert_BC_s}
  \end{align}
Here the limits appearing on the left hand side of~\eqref{Eq_PseudoholomPlaneWithHamPert_BC_t}  and~\eqref{Eq_PseudoholomPlaneWithHamPert_BC_s} are understood in the $C^0(\R)$--topology and $a \le b$ are arbitrary.  Notice  that~\eqref{Eq_PseudoholomPlaneWithHamPert_mod} is the pseudoholomorphic map equation with a Hamiltonian perturbation. Namely, the time-dependent Hamiltonian function here is $\im{(e^{-i\theta_\nu(t)}f)}$. 

Notice also that it is assumed  that the integral in~\eqref{Eq_PseudoholomPlaneWithHamPert_BC_t} is convergent for all $t\in\R$. For instance, this is the case if $\partial_su\in W^{k,p}(\R^2; u^*TM)$ with $k>\max \{\tfrac 1p,\tfrac 2p -1\}$. In this case, we have

\[\label{Eq_AuxEstimForTrace}
\| \partial_su (\cdot, \tau)\|_{W^{0,1}(\R)}\le C_{k,p}\| \partial_s u \|_{W^{k,p}(H_\tau)},
\]
where $H_\tau=\{t\ge\tau\}\subset\R^2$. In particular, $\int_{-\infty}^{+\infty} |\partial_s u(s,t)|\, ds$ tends to zero as $t\to+\infty$ and similarly for $t\to -\infty$.

It is very instructive to see a relation between solutions of~\eqref{Eq_PseudoholomPlaneWithHamPert_mod}--\eqref{Eq_PseudoholomPlaneWithHamPert_BC_s} and
pseudoholomorphic strips as in Seidel's approach. This is outlined in
Appendix~\ref{Apend_PseudoholomStripsAndPlanes}. However, instead of proving that such a connection indeed holds, we study equations~\eqref{Eq_PseudoholomPlaneWithHamPert_mod}--\eqref{Eq_PseudoholomPlaneWithHamPert_BC_s}  directly, since in view of the intended applications it is important to have direct proofs of the basic properties (compactness, Fredholm property, transversality  etc.). In this paper we prove compactness and Fredholm property for solutions of~\eqref{Eq_PseudoholomPlaneWithHamPert_mod}--\eqref{Eq_PseudoholomPlaneWithHamPert_BC_s}.

\medskip

Next we show how to define the map $\mu^2$ in our framework. Let $\Om$ be a (non-compact) Riemann surface containing
three ``long necks''. By this we mean a triple of holomorphic embeddings%
\begin{equation*}
 \imath_1,\imath_2\colon \{ z\; |\ \re z < 0  \} \rightarrow\Om,\qquad \text{and}\qquad  
\imath_3\colon \{ z\; |\ \re z > 0  \} \rightarrow\Om%
\end{equation*}
with disjoint images. To be more explicit,  we choose the complex plane $\mathbb C$ as a model for $\Om$ (see
Fig.~\ref{Fig_2dDomain}), where the embedding $\imath_1$ is given in polar coordinates by $(\varrho,\varphi)\mapsto (\varrho^{2/3},\tfrac
23(\varphi +\pi)),\ \tfrac \pi 2< \varphi <\tfrac {3\pi}2$ and the other two embeddings are defined similarly. The
curves shown on the figure are of the form $t\mapsto \imath_j(s,t)$. This is  our analogue of the ``pair of pants''
surface. 
\begin{figure}[ht]
   \begin{center}
    \includegraphics[width=0.4\textwidth]{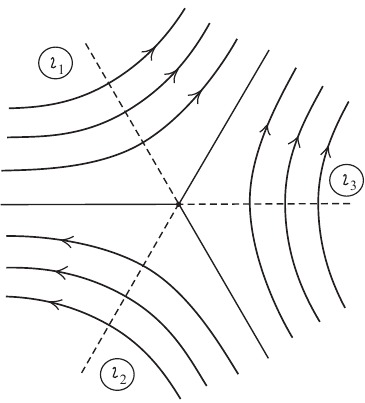}
   \end{center}
\caption{ The domain $\Om$ with three long necks.}
\label{Fig_2dDomain}
\end{figure}

Further, pick any three critical points, say $m_1,m_2, m_3$  and a pair $(\gamma_1,\gamma_2)$  of broken flow lines. 
More precisely, $\gamma_1$ and $\gamma_2$ are solutions of the equations%
\begin{equation*}
  \begin{aligned}
     &\ \;\dot \gamma_j +\cos\theta_{j,\nu}(t)\, v_0 +\sin\theta_{j,\nu}(t)\, v_1=0,\\
     &\lim\limits_{t\to +\infty} \gamma_j(t) =m_j,\quad \lim\limits_{t\to -\infty} \gamma_j(t) =m_{j+1}.
  \end{aligned}
\end{equation*}
Here $\theta_{j,\nu}(t)$ is a perturbation of the function obtained from $\theta_0(t)$ by putting  $(\theta_-,\theta_+)=(\arg (z_j-z_0),\; \arg(z_{j+1}-z_0))$.
Then $\mu^2(\gamma_1,\gamma_2)$  should be a formal linear combination of broken
 flow  lines connecting $m_1$ with $m_3$. Pick any such flow line, i.e., a solution of the
problem\footnote{Our convention is that for $m_-<m_+$ a broken flow line goes from $m_+$ to $m_-$ as $t$
varies between $-\infty$ and $+\infty$ and therefore the asymmetry between $\gamma_3$ and $\gamma_1,\gamma_2$. }%
\begin{equation*}
  \begin{aligned}
     &\ \;\dot \gamma_3+\cos\theta_{3,\nu}(t)\, v_0 +\sin\theta_{3,\nu}(t)\, v_1=0,\\
     &\lim\limits_{t\to +\infty} \gamma_3(t) =m_1,\quad \lim\limits_{t\to -\infty} \gamma_3(t) =m_{3},
  \end{aligned}
\end{equation*}
and also choose  $\eta\in\Om^{0,1}(\Om)$ such that for $j=1,2,3$ we have $\imath_j^*\eta=\tfrac 12e^{i\theta_{j,\nu}(t)}d\bar z$ provided  $\bigl |\re z\bigr |\ge 1$. 
Then the multiplicity of $\gamma_3$ can  conjecturally be defined by counting solutions of the equations%
  \begin{align}
     &\ \;\bar\partial u +\eta\otimes v_0(u)=0, &  \qquad & u\colon\Om\rightarrow M,\label{Eq_Mu2_1}\\
  &\lim\limits_{t\to \pm\infty} u\comp\imath_j (s,t) =m_{\sigma_\pm(j)},\  & &
       \lim\limits_{t\to\pm\infty} \int_{0}^{\infty}\bigl | \partial_s\bigl (u\circ\imath_j(s,t) \bigr ) \bigr |\, ds=0,\quad  j=1,2,3, \label{Eq_Mu2_2}\\    
&\lim\limits_{s\to \infty} u\comp\imath_j (s,t) =\gamma_j(t)\    & & \lim\limits_{s\to\infty}\int_{a}^b\bigl | \partial_s\bigl (u\circ\imath_j(s,t) \bigr ) \bigr |\, dt=0,\quad  j=1,2,3.\label{Eq_Mu2_3} 
     \end{align}
Here $\eta\otimes v_0(u)\in\Om^{0,1}(\Om; u^*TM)$, $\sigma_+(1,2,3)=(1,2,1)$, $\sigma_-(1,2,3)=(2,3,3)$. Moreover, in~\eqref{Eq_Mu2_3} $"s\to\infty"$ means $s\to -\infty$ for $j=1,2$ and $s\to +\infty$ for $j=3$; The meaning of $"\infty"$ in~\eqref{Eq_Mu2_2} is similar.

 Notice that over the long necks the above equations and equations~\eqref{Eq_PseudoholomPlaneWithHamPert_mod}--\eqref{Eq_PseudoholomPlaneWithHamPert_BC_s} are of a similar form.

The analogue of holomorphic discs with $d+1$ punctures on the boundary involved in the definition of $\mu^d$ are defined
in a similar manner.


\medskip

Let us briefly summarize. We can conjecturally associate with $(f,J)$ a directed $A_\infty$-category $\mathcal A(f, J)$ as follows. The objects of $\mathcal A(f, J)$ are critical points of
$f$. For any pair  $(m_-, m_+)$ of critical points,  denote by $CF(m_-, m_+)$ the  vector space generated by $\Gamma_\nu(m_-; m_+)$ and put %
\begin{equation*}
  hom_{\mathcal A(f,J)}(m_-, m_+)=%
  \begin{cases}
    CF(m_-, m_+)  & m_-< m_+,\\
    \mathbb Z/2\cdot id  & m_-=m_+,\\
    0  &m_->m_+.
  \end{cases}
\end{equation*}
For $\gamma_\pm\in \Gamma_\nu(m_-; m_+)$ denote by $\mathcal M_\nu^0(\gamma_-,\gamma_+)$ the zero-dimensional component of the space %
$\{ u\; |\; u\ \text{solves~\eqref{Eq_PseudoholomPlaneWithHamPert_mod}--\eqref{Eq_PseudoholomPlaneWithHamPert_BC_s}} \}/\mathbb R$. Assuming $\#\mathcal M_\nu^0(\gamma_-,\gamma_+)$ makes sense, we can define $\mu^1$ by declaring
\begin{equation*}
 \mu^1(\gamma_-)=\sum_{\gamma_+} \bigl (\#\mathcal M_\nu^0(\gamma_-,\gamma_+)\!\!\!\!\mod 2\bigr )\, \gamma_+.
\end{equation*}
The maps $\mu^d$ for $d\ge 2$ are defined in a similar manner and together with $\mu^1$ (conjecturally) combine to an
$A_\infty$-structure. Clearly, $\mathcal A(f,J)$ depends on the various choices involved in the construction. However,
as explained  in~\cite{Seidel:01VanishingCyclesI} the derived category $D^b(\mathcal A(f,J))$ should not depend on these
choices. Moreover, assume $(f_\tau, J_\tau),\ \tau\in [0,1]$ is a continuous family such that $f_\tau$ is a
$J_\tau$-holomorphic function, whose critical points lie in pairwise different fibres for all $\tau$. Then 
$D^b(\mathcal A(f_0,J_0))$ is equivalent to $D^b(\mathcal A(f_1,J_1))$.

\begin{rem}\label{Rem_FSCatForMultivaluedFns}
Our main example is the complex Chern--Simons functional, which takes values in $\mathbb C/\mathbb Z$ rather than in $\mathbb C$. In this case, the construction outlined above does not immediately apply. However, we may proceed as follows. Assume that each line $\ell_r=\{z\; |\;\re z=r\mod\mathbb Z\}$ contains at most one critical value of $f$ (possibly after a perturbation). Pick $r$ such that the line $\ell_r$ does not contain any critical value of $f$ and ``cut" the cylinder  $\mathbb C/\mathbb Z$  along $\ell_r$ to obtain a holomorphic function $f_r$ with values in $(0,1)\times\mathbb R$. In other words, consider only those flow lines $\gamma$ of $f$ for which the image of $f\comp\gamma$ does not intersect the line $\ell_r$.  Then $D^b(\mathcal A(f_r))$ does not depend on $r$ as long as $r$ varies in a connected interval $I$ such that $I\times\mathbb R$ does not contain any critical value of $f$. In this way we obtain a collection of $k$ triangulated categories $\bigl ( D^b(\mathcal A(f_{r_j})) \bigr )_{j=1}^k$, which is well-defined up to a cyclic permutation. Here $k$ is the number of critical values of $f$.
\end{rem}

 \subsection{A priori $C^0$-estimates}\label{Subsect_AprioriC0Estim}

Since $M$ is not compact, we need to show that solutions of~\eqref{Eq_PseudoholomPlaneWithHamPert_mod}--\eqref{Eq_PseudoholomPlaneWithHamPert_BC_s} do not leave a fixed compact subset of $M$. This is proved in this subsection under an additional assumption. 

\medskip

The proof of  Theorem~\ref{Thm_AprioriC0Estimates}, which is  the main result of this subsection, crucially depends on the choice of the perturbation $\theta_\nu$ of the function~\eqref{Eq_FunctionTheta0}. So we  take a moment to describe  the missing details.
 
Just like in the beginning of the previous subsection fix a pair of critical points $(m_-,m_+)$ and put $z_\pm=f(m_\pm)$. Up to a translation and a rotation we can assume that $z_0=0,\ \theta_\pm\in (0,\pi), \im z_-=\im z_+=\zeta>0$. For $\nu\in (0,1)$ consider a smooth function $\theta_\nu\colon\R\to\R$, which satisfies
\[
\theta_\nu (t)=
\begin{cases}
  0\qquad &|t|\ge \nu^{-1}+1,\\
 \theta_+\qquad & t\in [-\nu^{-1}, -\nu]\\
 \theta_- -\pi\qquad & t\in [\nu, \nu^{-1}] 
\end{cases}
\]  
and is monotone on the intervals $(-\nu^{-1}-1, -\nu^{-1}),\ (-\nu, \nu),$ and $(\nu^{-1}, \nu^{-1}+1)$. We also assume that $\theta_\nu(t)\ge 0$ for $t\le 0$ and that $\theta_\nu(t)\le 0$ for $t\ge 0$. The graph of $\theta_\nu$ is shown on Fig.~\ref{Fig_GraphThetaNu}.

\begin{figure}[ht]
\centering 
   \includegraphics[width=0.7\textwidth]{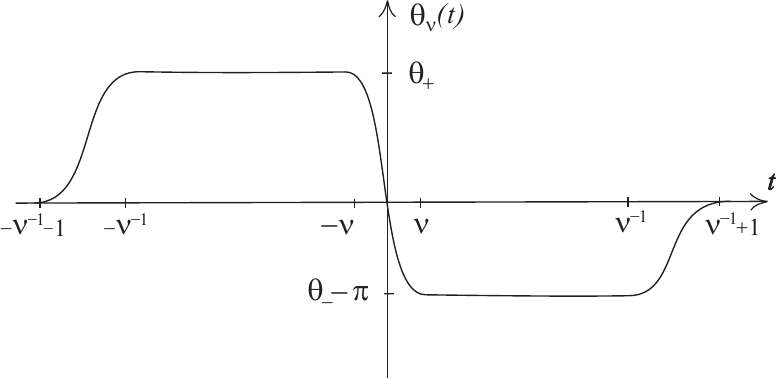}
\caption{Graph of $\theta_\nu$.}
\label{Fig_GraphThetaNu}
\end{figure}

\begin{prop}\label{Prop_ConvergentSubseqOfBrokenFlowLines}
 Suppose the closed domain $G$ bounded by the triangle $z_-z_0z_+$   contains no critical values of $f$ other than $z_\pm$. Let $\nu_j\in (0,1)$ and $\gamma_j\in\Gamma_{\nu_j}(m_-, m_+)$ be arbitrary sequences such that $\nu_j\to 0$.  Then there exists a subsequence  $j_k\to\infty$ such that $\gamma_{j_k}$ converges in $C^0(\mathbb R; M)$ and $\gamma_0=\lim\limits_{k\to\infty}\gamma_{j_k}$ is a solution of~\eqref{Eq_BrokenFlowLine}.
\end{prop}

The proof of Proposition~\ref{Prop_ConvergentSubseqOfBrokenFlowLines} is given in Appendix~\ref{App_OnThePertAntigradFlowLines}.  The assumption that $G$ contains no critical values of $f$ other than $z_\pm$ is essential. If this assumption does not hold, we may have solutions of~\eqref{Eq_PertBrokenFlowLine}, such that the curve $f\circ\gamma_\nu$ is not even homotopic (relative endpoints) to $\ell=\overline{z_+z_0}\cup\overline{z_0z_-}$ in $\C\setminus\{z_1,\dots, z_m\}$ (see also Lemma~\ref{Lem_DistToBrokenLine}). In particular, the conclusion of Proposition~\ref{Prop_ConvergentSubseqOfBrokenFlowLines} is false in this case.   To exclude  such phenomena, we make the following additional assumption.

\begin{hypothesis}\label{Hyp_ConvPosition}
     \textit{Convex position of critical values.} The critical values of $f$ are in convex position, i.e., none of the critical values of $f$ is contained in the convex hull of the other critical values. The base point $z_0$ lies in the interior of the convex hull of the critical values.
\end{hypothesis}



\begin{thm}\label{Thm_AprioriC0Estimates}$\phantom{t}$
There exists a compact subset $\hat K\subset M\setminus\partial M$ such that the image of any solution of~\eqref{Eq_PseudoholomPlaneWithHamPert_mod} satisfying
\begin{equation}\label{Eq_BCwithoutIntegrals}
\lim\limits_{t\to\pm\infty} u(s,t)= m_\mp, \qquad \lim\limits_{s\to\pm\infty} u(s,t)=\gamma_\mp(t)
\end{equation}
is contained in $\hat K$.
\end{thm}
\begin{proof}
The proof consists of the following two steps.

\setcounter{thestep}{0}
\begin{step}\label{Step_BoundOnFU_Thm_AprioriC0Estimates}
  There exists a constant $\hat R>0$ such that for any solution $u\in C^2(\R^2)$ of~\eqref{Eq_PseudoholomPlaneWithHamPert_mod}  and~\eqref{Eq_BCwithoutIntegrals}  
 we have 
\[
\bigl |f\circ u(s,t)\bigr |\le \hat R\qquad\text{for all } (s,t)\in\R^2.
\]  
\end{step}

Denote $f\comp u=\varphi +i\psi$ and observe that Floer's equation for $u$ implies the equations
\begin{equation}\label{Eq_AuxPertCauchyRiemann}
\partial_s\varphi -\partial_t\psi =\sin\theta_\nu (t)\, \rho\circ u,\quad \partial_s\psi +\partial_t\varphi =-\cos\theta_\nu (t)\, \rho\circ u.
\end{equation}

 Denote 
\begin{equation*}
  \Theta_1(t)=\frac 1r\Bigl (\int_0^t\cos\theta_\nu(\tau)\, d\tau -t\Bigr ),\qquad
 \Theta_2(t)=\frac 1r\int_0^t\sin\theta_\nu(\tau)\, d\tau
\end{equation*}
and notice that $\Theta_1$ and $\Theta_2$ are bounded both from above and below (in fact, $\Theta_i(t)$ is locally constant for $|t|\ge \nu^{-1}+1$). This crucial property is a corollary of our particular choice of $\theta_\nu$. 

Put  $\overline\Theta_j=\sup_{\R}\Theta_j (t),\ \underline\Theta_j=\inf_{\R}\Theta_j (t),\  j=1,2$. Furthermore, choose $R>0$ so large that $f(K)\subset B_R(0)$, where $K$ is the compact subset in (P\ref{Pty_LocTrivAwayFromK}).  We claim that the following inequality
\begin{equation}\label{Eq_AuxIneqForSupPhi}
\sup_{\R^2}\bigl ( \varphi(s,t)+\Theta_1(t)  \bigr )\le R +\overline\Theta_1
\end{equation}
holds for all $(s,t)\in\R^2$. We argue by contradiction. Indeed,  assume $\varphi (s_0,t_0)+\Theta_1(t_0)=\sup \bigl ( \varphi(s,t)+\Theta_1(t)  \bigr )>R+\overline\Theta_1$ for some $(s_0,t_0)\in \R^2$ (the boundary conditions for $u$ imply that the supremum must be attained at some point in $\R^2$). Then $\varphi(s_0, t_0)>R$ so that $(\varphi,\psi)\in B_R^c(0)$ for all $(s,t)$ lying in some neighbourhood $U$ of $(s_0,t_0)$.  Since $\rho=r^{-1}$ everywhere on $f^{-1}(B_R^c(0))$, from~\eqref{Eq_AuxPertCauchyRiemann} we obtain 
\[
\Delta\varphi =r^{-1}\theta_\nu'(t)\sin\theta_\nu(t), \qquad (s,t)\in U.
\]  
Hence, the function $\varphi +\Theta_1$ is harmonic in $U$ and achieves its  maximum at $(s_0, t_0)\in U$.  This contradiction proves~\eqref{Eq_AuxIneqForSupPhi}.

Inequality~\eqref{Eq_AuxIneqForSupPhi} implies in turn the estimate
\[
 \sup_{\R^2}\varphi(s,t)\le R + (\overline\Theta_1-\underline\Theta_1).
\]
Arguing along similar lines one also obtains
\begin{equation*}
  \begin{gathered}
  \inf_{\R^2}\varphi(s,t) \ge -R - (\overline\Theta_1-\underline\Theta_1),   \\
  \sup_{\R^2} \psi(s,t)\le R + (\overline\Theta_2-\underline\Theta_2), \qquad
\inf_{\R^2} \psi(s,t)\ge -R - (\overline\Theta_2-\underline\Theta_2). 
  \end{gathered}
\end{equation*}
This finishes the proof of Step~\ref{Step_BoundOnFU_Thm_AprioriC0Estimates}.

\begin{step}
We prove the theorem.
\end{step}

Let $W\supset\partial M,\ W'\supset\partial M\times\C$, and $\psi\colon W'\to W$ be as in the paragraph following~(P\ref{Pty_LocTrivAwayFromK}). Observe that property~(P\ref{Pty_ConvexBoundary}) implies that the boundary of $M_0$ is $J_{M_{0}}$--convex, i.e., there exists a function $h\colon M_0\to (-\infty, 0]$, which is plurisubharmonic in a neighbourhood of the boundary and $\partial M_0=h^{-1}(0)$. Choose $\e>0$ so small that $h$ is subharmonic on $h^{-1}(-\e,0)$ and $U'=h^{-1}(-\e,0)\times B_R(0)\subset W'$. Denote $U=\psi(U')$. 

We claim that for any solution $u$ of~\eqref{Eq_PseudoholomPlaneWithHamPert_mod}  and~\eqref{Eq_BCwithoutIntegrals} we have $u(\R^2)\cap U=\varnothing$. Indeed, assuming the converse, there exists $z_0=(s_0, t_0)$ such that $h\comp u(z_0)=\sup \{ h\comp u(z)\mid u(z)\in U\}$. Then for sufficiently small $\delta>0$ we can think of $u$ as a map $B_\delta(z_0)\to M_0\times\C$. If $\pi_1$ denotes the projection to the first components, the map $\pi_1\comp u$ is pseudoholomorphic. Moreover, $h\comp\pi_1\comp u=h\comp u$ has a local maximum at $z_0$, which is a contradiction. 

Thus  the image of $u$ is contained in $\hat K=f^{-1}(B_R(0))\setminus U$. It remains to notice that $\hat K$ is compact. 
%
\end{proof}

\begin{rem}
  We would like to stress that other results in this paper (except those in Appendix~\ref{App_OnThePertAntigradFlowLines}) depend on hypothesis~(H\ref{Hyp_ConvPosition}) only through Theorem~\ref{Thm_AprioriC0Estimates}. It is quite possible that an a priori $C^0$--bound can still be proved for a different choice of the perturbation $\theta_\nu$, which does not require convex position of the critical values. However at present it is not quite clear how to obtain such an estimate without~(H\ref{Hyp_ConvPosition}). 
\end{rem}



\subsection{The action functional and the energy identity}
Denote 
\[
\begin{aligned}
W_{m_-,m_+}^{2,2}=\bigl\{ \gamma\in W_{loc}^{2,2}(\R; M)\mid 
&\text{ there exist } T>0\ \text{and } \xi_\pm\in W^{2,2}\bigl ((T,\infty); T_{m_\pm}M\bigr )\bigr.\\
&\bigl.\ \text{s.t. } \gamma (\pm t)=\exp_{m_\pm}\xi_\pm(t)\text{ for } t>T \bigr\}.
\end{aligned}
\]
Then the action functional
\begin{equation}\label{Eq_ActionFunctional}
  \mathscr F(\gamma) = \int\limits_{\R}\gamma^*\lambda 
+\int\limits_{\R}\im{\bigl (e^{-i\theta(t)}f\circ\gamma (t) \bigl )}\, dt
\end{equation}
is well--defined as a map $\mathscr F\colon W_{m_-,m_+}^{2,2}\to\R$. Indeed the first integral is convergent, since $\gamma^*\l\in W^{1,2}(\R)\hookrightarrow L_1(\R)$. As for the second integral, the convergence follows from the fact that $f$ is a quadratic function in an appropriate coordinate chart at $m_\pm$. Observe also, that  $\mathscr F$ is essentailly the standard symplectic action functional with a Hamiltonian perturbation. 

Consider the time-dependent vector field 
\[
v^t= \grad\re{\left ( e^{-i\theta_\nu(t)}f \right)} 
=\cos\theta_\nu(t)\, v_0 +\sin\theta_\nu(t)\, v_1.
\]
A standard computation shows that $d\mathscr F (\xi)= -\int_{\mathbb R}\om(\xi, \dot\gamma + v^t)\, dt$, where  $\xi$ is a vector field along $\gamma$. Here we used the fact, that the symplectic gradient of $f_0$ is $v_1=\grad  f_1$. 
Therefore with respect to the $L^2$--metric we have $ \grad \mathscr F = J(\dot\gamma + v^t)$. Hence,  the critical points of the functional $\mathscr F$ 
are broken  flow lines of $f$ connecting $m_+$ and $m_-$.  Similarly, the antigradient flow lines of $\mathscr F$ can be interpreted as  solutions of equations~\eqref{Eq_PseudoholomPlaneWithHamPert_mod}.



Define the energy of a solution $u$ of~\eqref{Eq_PseudoholomPlaneWithHamPert_mod} by
 \begin{equation*}
 E(u)=\frac12\int_{\mathbb R^2} \bigl ( |\partial_s u|^2 + |\partial_t u+v^t|^2 \bigr )ds\wedge dt=\int_{\mathbb R^2}  |\partial_s u|^2\, ds\wedge dt.
\end{equation*}

\begin{thm}[Energy identity]\label{Thm_EnergyId}
  Let $u\in C^1(\R^2; M)$ be a solution of equations~\eqref{Eq_PseudoholomPlaneWithHamPert_mod}--\eqref{Eq_PseudoholomPlaneWithHamPert_BC_s}. Then  
\[
E(u)=\mathscr F(\gamma_+)-\mathscr F(\gamma_-).
\]
In particular, $E(u)<\infty$.
\end{thm}
\begin{proof}
  It is convenient to denote $\beta_t(s)=u(s,t)=\gamma_s(t)$. Pick arbitrary positive numbers $\sigma$ and $\tau$. Using Stokes' theorem and the identity
\[
\om(v^t, \partial_su)= \frac \partial{\partial s}\im{\bigl (e^{-i\theta(t)}f\circ u (s,t)\bigr )}
\] 
a standard computation yields
\begin{equation}\label{Eq_AuxEnergyId}
 \begin{aligned}
  \int\limits_{-\tau}^\tau\int\limits_{-\sigma}^\sigma |\partial_s u|^2\, ds\wedge dt &=
\int\limits_{-\tau}^\tau\int\limits_{-\sigma}^\sigma \om\bigl (\partial_s u, \partial_tu +v^t\bigr )\, ds\wedge dt\\
&=\int\limits_{-\tau}^{\tau} \gamma_{-\sigma}^*\l
- \int\limits_{-\tau}^{\tau}\gamma_{\sigma}^*\l
+\int\limits_{-\sigma}^{\sigma} \beta_{\tau}^*\l
-\int\limits_{-\sigma}^{\sigma} \beta_{-\tau}^*\l\\%
&-\int\limits_{-\tau}^{\tau}\im{e^{-i\theta(t)}f\circ\gamma_\sigma (t)}\, dt
+ \int\limits_{-\tau}^{\tau}\im{e^{-i\theta(t)}f\circ\gamma_{-\sigma} (t)}\, dt.
\end{aligned}
\end{equation}
With the help of equation~\eqref{Eq_PseudoholomPlaneWithHamPert_mod} we obtain
\[
\l (\partial_tu)-\l (\dot\gamma_\pm)=\l (J\partial_s u) +\l \bigl (v^t (\gamma_\pm)\bigr )-\l \bigl (v^t(u)\bigr ).
\] 
This in turn implies by~\eqref{Eq_PseudoholomPlaneWithHamPert_BC_s} that 
\[
  \int_{-\tau}^{\tau}\gamma_{\pm\sigma}^*\l \longrightarrow \int_{-\tau}^{\tau}\gamma_{\mp}^*\l\quad\text{as}\ \sigma\to +\infty.
\] 
Similarly, by~\eqref{Eq_PseudoholomPlaneWithHamPert_BC_t} we also have
\[
 \int_{-\infty}^{+\infty}\beta_{\pm\tau}^*\l \longrightarrow 0,\quad\text{as}\ \tau\to +\infty.
\]
Hence, passing in~\eqref{Eq_AuxEnergyId} first to the limit as $\sigma\to +\infty$ and then to the limit as $\tau\to +\infty$ we obtain the statement of the theorem.  
\end{proof}


\subsection{A priory $C^\infty$--estimates}

It is convenient to introduce the $L^p$--version of the energy of a map $u$:
\[
E_p(u)=\frac12\int_{\mathbb R^2} \bigl ( |\partial_s u|^p + |\partial_t u+v^t|^p \bigr )ds\wedge dt.
\]
In particular, $E(u)=E_2(u)$.

\begin{prop}\label{Prop_UnifConvergenceOfDsU}
 Let $u$ be a solution of~\eqref{Eq_PseudoholomPlaneWithHamPert_mod} with $E_p(u)<\infty$ for some $p\in [2,\infty)$. Then the following holds:
\begin{alignat}{3}
 &\lim\limits_{s\to\pm\infty} \partial_s u(s,t)=0\qquad &&\text{and}\qquad &&\lim\limits_{t\to\pm\infty} \partial_s u(s,t)=0;\label{Eq_LimitsDsU}\\
 &\,\ \sup\limits_{\mathbb R^2}|\partial_su|<\infty          &&\text{and} && \,\ \sup\limits_{\mathbb R^2}|\partial_tu|<\infty.\label{Eq_SupremumDU}
\end{alignat}
Here both limits in~\eqref{Eq_LimitsDsU}   are understood in the $C^0(\mathbb R)$--topology.
\end{prop}

\begin{proof}
We have~\cite[p.12]{Salamon:90MorseTheory} the following local estimate
\begin{equation*}
|\partial_su(s,t)|^2\le \frac 8{\pi r^2}\int_{B_r(s,t)}|\partial_s u|^2 +cr^2
\end{equation*}
provided $ \int_{B_r(s,t)}|\partial_s u|^2<h$. Here the constants $h, c>0$ depend on $M, J,\om$ and $f$ but not on $u$. Applying the estimate $\|\partial_su\|_{L^2(B_r(s,t))}\le \bigl (vol(B_r(s,t))\bigr)^{\frac 12 -\frac 1p}\|\partial_su\|_{L^p(B_r(s,t))}$ we obtain
\begin{equation}\label{Eq_IneqC0ThroughLp}
|\partial_su(s,t)|^2\le \frac 8{(\pi r^2)^{\frac 2p}} \|\partial_su\|^2_{L^p(B_r(s,t))} +cr^2
\end{equation}
provided $\|\partial_su\|_{L^p(B_r(s,t))}\le h\pi^{\frac 1p-\frac 12}$ and $r\le 1$. 

Pick an arbitrary $\e\in (0,1)$ and put $r=\sqrt\e$. Since by assumption $\|\partial_su\|^2_{L^p(\R^2)}<\infty$ there exists $R_\e>0$ such that $ \|\partial_su\|^2_{L^p(B_1(s,t))}< \e^{1+\frac 2p}$ provided $\max\{|s|, |t| \}>R_\e$. From~\eqref{Eq_IneqC0ThroughLp} we obtain $|\partial_su(s,t)|^2\le (8\pi^{-\frac 2p}+ c)\e$, which proves~\eqref{Eq_LimitsDsU}.


Further,  the first inequality in~\eqref{Eq_SupremumDU} follows immediately from~\eqref{Eq_LimitsDsU}. The second  inequality in~\eqref{Eq_SupremumDU}  is obtained from the first one using equation~\eqref{Eq_PseudoholomPlaneWithHamPert_mod} and the fact that $|v^t|^2=\rho$ is  bounded.
\end{proof}



\begin{cor}
   Let $u$ be a solution of~\eqref{Eq_PseudoholomPlaneWithHamPert_mod} with $E_p(u)<\infty$ for some $p\in [2,\infty)$. Then the convergence in~\eqref{Eq_BCwithoutIntegrals} in the $C^0$--topology implies the convergence in the $C^1$--topology. 
\end{cor}

\begin{lem}\label{Lem_LocalEstimForCkNorm}
  Let $\Om$ be a bounded domain in $\mathbb R^2$. For any integer $k\ge 2$ and any $c_1>0$ there exists $c_k=c_k(c_1,\Om)$ with the following significance. For any solution $u$ of~\eqref{Eq_PseudoholomPlaneWithHamPert_mod} the following implication holds:
  \begin{equation*}
    \sup\limits_{\mathbb R^2} |\partial_s u|\le c_1\qquad\Longrightarrow\qquad   \| u\|_{C^k(\Om)}\le c_k.
  \end{equation*}
\end{lem}

The proof of this lemma relies  on the local properties of solutions of Floer's equation and  can be obtained along the same lines as the proof of Lemma~C.3 in~\cite{RobbinSalamon:01_AsymptBehOfHolStrips} (in fact the argument simplifies as we do not need to consider charts with Lagrangian boundary conditions). We omit the details.


\begin{prop}\label{Prop_BoundInC1ImpliesBoundInCk}
  For any integer $k\ge 2$ and any $c_1>0$ there exists $c_k=c_k(c_1)$ with the following significance. For any solution $u$ of~\eqref{Eq_PseudoholomPlaneWithHamPert_mod} the following implication holds:
  \begin{equation*}
    \sup\limits_{\mathbb R^2} |\partial_s u|\le c_1\qquad\Longrightarrow\qquad   \| u\|_{C^k(\R^2)}\le c_k.
  \end{equation*}
\end{prop}
\begin{proof}
  From  Lemma~\ref{Lem_LocalEstimForCkNorm} we obtain that there exists a constant $c_k$ such that 
\begin{equation}\label{Eq_AuxCkEstimForU}
   \| u\|_{C^k(\bar\Om)}\le c_k,
  \end{equation}
where $\Om=(0,1)\times (-\nu-2,\nu +2)$. This implies that estimate~\eqref{Eq_AuxCkEstimForU} is valid for $\Om=\mathbb R\times(-\nu-2,\nu +2)$ since equation~\eqref{Eq_PseudoholomPlaneWithHamPert_mod} is invariant with respect to shifts in the $s$--variable.
Applying Lemma~\ref{Lem_LocalEstimForCkNorm} to $\Om=(0,1)\times (\nu+1,\nu+2)$  and observing that both $J$ and $v^t$  depend neither on $s$ nor on $t$ provided $t\ge \nu+1$ we obtain that  estimate~\eqref{Eq_AuxCkEstimForU} also  holds for $\Om=\R\times (\nu+1,+\infty)$. Similarly, estimate~\eqref{Eq_AuxCkEstimForU} is valid for $\Om=\R\times (-\nu-1,-\infty)$ as well. This clearly implies the statement of the proposition.
\end{proof}



\begin{thm}
For any integer $k\ge 0$ there exists a constant $c_k>0$ such that for any solution $u$ of~\eqref{Eq_PseudoholomPlaneWithHamPert_mod},\eqref{Eq_BCwithoutIntegrals} with $E_p(u)<\infty$ for some $p\in [2,\infty)$  we have
  \begin{equation*}
   \| u\|_{C^k(\mathbb R^2)}<c_k. 
  \end{equation*}
Here constants $c_k$ depend on $M, J,$ and $f$ but not on $u$.
\end{thm}
\begin{proof}
  It follows from  Proposition~\ref{Prop_UnifConvergenceOfDsU} that for any solution $u$ of~\eqref{Eq_PseudoholomPlaneWithHamPert_mod} with $E_p(u)<\infty$  we have 
\[
\|\nabla u\|_{L_\infty}=\max\,\bigl\{\sup_{\R^2}|\partial_su|,\ \sup_{\R^2}|\partial_t u| \bigr\}<\infty.
\]
Since the symplectic form is exact, the  bubbling phenomenon does not happen. This in turn implies that 
\[
c_1=\sup_{u\in\mathcal M(\gamma_-,\gamma_+)}\| \nabla u\|_{L_\infty}<\infty.
\]
The rest follows immediately from Proposition~\ref{Prop_BoundInC1ImpliesBoundInCk}.
\end{proof}

\begin{cor}\label{Cor_AprioriEstimForCkNorms}
For any integer $k\ge 0$ there exists a constant $c_k>0$ such that for any solution $u$ of~\eqref{Eq_PseudoholomPlaneWithHamPert_mod}--\eqref{Eq_PseudoholomPlaneWithHamPert_BC_s}  we have
  \begin{equation*}
   \| u\|_{C^k(\mathbb R^2)}<c_k. 
  \end{equation*}
\end{cor}


\subsection{Asymptotic behaviour}

Pick any smooth curve $\gamma\colon\R\to M$ such that $\gamma(t)\to m_\pm$ as $t\to\mp\infty$ and denote
\begin{equation}\label{Eq_Map_sigma}
\sigma (\gamma)=\sigma_\nu(\gamma)=\dot\gamma +v^t(\gamma)\in\Gamma(\gamma^*TM). 
\end{equation}
Obviously, $\sigma(\gamma)=0$ if and only if $\gamma$ is a broken flow line of $f$. Consider the linearisation of $\sigma$ at the point $\gamma$:
\[
D_\gamma\sigma(\eta)=\nabla_t\,\eta +\nabla_\eta v^t,\qquad \eta\in\Gamma(\gamma^*TM).
\] 

From now on we assume that all broken flow lines of $f$ are generic. To be more precise, we assume that the following hypothesis holds. 

\begin{hypothesis}\label{Hyp_NondegOfBrokenFlowLines}
  \textit{Nondegeneracy of broken flow lines.} All solutions of~\eqref{Eq_PertBrokenFlowLine} are nondegenerate in the following sense: The operator 
  \begin{equation}\label{Eq_DsigmaW12L2}
    D_\gamma\sigma\colon W^{1,2}(\gamma^*TM)\longrightarrow L^2(\gamma^*TM)
  \end{equation}
is an isomorphism.
\end{hypothesis}

\begin{rem}
  It is proved in Appendix~\ref{App_OnThePertAntigradFlowLines} that hypothesis~(H\ref{Hyp_NondegOfBrokenFlowLines}) holds provided the vanishing cycles corresponding to the segments $\overline{z_0z_\pm}$ intersect transversely in $M_0$. 
\end{rem}


Similarly, pick a smooth map $u\colon\R^2\to M$ satisfying boundary conditions~\eqref{Eq_BCwithoutIntegrals} and denote $\Sigma(u)=\partial_s u + J\bigl (\partial_t u + v^t(u)\bigr )\in\Gamma(u^*TM)$. Consider the linearisation of $\Sigma$ at the point $u$:
\begin{equation}\label{Eq_LinerizationOfHolPlane}
  \begin{aligned}
   D_u\Sigma(\xi) &=\cd s\xi + J\bigl ( \cd t\xi +\cd \xi v^t  \bigr ) + \nabla_\xi J (\partial_t u + v^t),\\%
		    &= \cd s \xi+ J\cd t \xi +\cos\theta_\nu \cd\xi v_1 -\sin\theta_\nu \cd\xi v_0+ \cd \xi\! J (\partial_t u),
 \end{aligned}
\end{equation}
where $\xi\in\Gamma(\mathbb R^2;u^*TM)$.


\begin{rem}
  The maps $\sigma$ and $\Sigma$ can be viewed as sections of certain Banach bundles (see pp.\pageref{Eq_AuxW12RM} and~\pageref{Eq_AuxDuSigmaBanachSpaces} for details). However this is not needed for the purposes of this subsection.
\end{rem}

It is convenient to  choose a unitary trivialization $\Psi$ of $u^*TM$. Recall that for each $(s,t)\in\R^2$ the map $\Psi(s,t)\colon \R^{2n}\to T_{u(s,t)}M$ is a linear isomorphism of complex Hermitian vector spaces, where $\R^{2n}$ is considered to be equipped with the standard complex structure and the standard symplectic form:
\[
J_0=
\begin{pmatrix}
  0 & -\mathbbm 1\\
\mathbbm 1 & \phantom{-}0
\end{pmatrix},\qquad\quad \om_0(\xi, \eta)=\xi^t J_0\eta,\quad \xi,\eta\in\mathbb R^{2n}.
\]
Also denote by $\psi_\pm$ the restriction of $\Psi$ to $\gamma_\pm$.

\begin{rem}\label{Rem_UnitaryTrivialization}
  One such trivialization can be constructed as follows. Choose a basis of $T_{m_-}M$ and trivialise $\gamma_-^*TM$ with the help of the parallel transport along $\gamma_-$. Then trivialise $u^*TM$ by doing parallel transport along the curves $\beta_t(\cdot)=u(\cdot, t)$.    
\end{rem}



With the help of the trivialisations chosen above we can think of the operators  $D_{\gamma_\pm}\sigma$ and $D_u\Sigma$ as acting on vector-valued functions. More precisely, there exist matrix-valued functions  $S(s,t)$ and $S_\pm(t)$ such that%
\begin{align*}
 & \Psi(s,t)\bigl ( \partial_s\xi + J_0\partial_t\xi + S(s,t)\xi  \bigr)=D_u\Sigma \bigl ( \Psi(s,t)\xi  \bigr ),\qquad & &\text{for all }\xi\in C^{\infty}(\mathbb R^2;\mathbb R^{2n});\\
 & \psi_\pm(t)\bigl ( \dot\eta -J_0 S_\pm(t)\eta  \bigr )=D_{\gamma_\pm}\sigma\bigl ( \psi_\pm(t)\eta  \bigr ), & &\text{for all }\eta\in C^{\infty}(\mathbb R;\mathbb R^{2n}).
\end{align*}
Explicitly, matrices $S$ and $S_\pm$ are given by the relations
\begin{align}
   \Psi(s,t)S(s,t)&=\nabla_s\Psi  + J\bigl ( \nabla_t\Psi + \nabla_{\Psi}v^t  \bigr ) +\nabla_{\Psi}J (\partial_tu + v^t),\label{Eq_MatrixS}\\
 \psi_\pm(t)S_\pm(t) &= J\bigl (\nabla_t\psi_\pm + \nabla_{\psi_\pm}v^t\bigr). \label{Eq_MatricesSpm}
\end{align}
To simplify the notations, denote also by $L$ and $l_\pm$ the operators representing $D_u\Sigma$ and $D_{\gamma_\pm}\sigma$ with respect to the chosen trivialisation:
\begin{equation}\label{Eq_LiniarizationInAFrame}
  L =\partial_s +J_0\partial_t +S(s,t),\qquad\quad l_\pm=\frac d{dt} -J_0S_\pm(t).
\end{equation}

\begin{lem}\label{Lem_ExpDecayForAnOperator}
  Assume the  following holds:
\begin{itemize}
  \item[(i)] $S\colon\R^2\to M_{2n}(\R)$ is $C^\infty$--bounded\footnote{this means that $S$ as well as all derivatives are bounded on $\R^2$};
  \item[(ii)] $S(s,t)$ converges to $S_\pm(t)$ in the $C^0(\R)$--topology as $s\to\mp\infty$;
 \item[(iii)]  The operators $l_\pm\colon W^{1,2}(\mathbb R;\mathbb R^{2n})\rightarrow L^{2}(\mathbb R;\mathbb R^{2n}) $ are invertible;
  \item[(iv)] $ \lim\limits_{s\to\pm\infty}\sup\limits_t \| \partial_sS(s,t)  \|=0$.
\end{itemize}
Let $\xi$ be a solution of the equation $D\xi=0$. If $\xi\in L^p(\R^2; \R^{2n})$ for some $p\in (1,+\infty)$, then there exist positive constants $C$ and $\delta$ such that
\begin{equation*}
  |\xi(s,t)|\le Ce^{-\delta |s|}\qquad \text{for all } (s,t)\in\R^2.
\end{equation*}
\end{lem}


\begin{proof}

Let $\xi\in L^p(\R^2; \R^{2n})$ be a solution of the equation $D\xi=0$. 
Since $D$ is $C^\infty$--bounded and uniformly elliptic, $\xi$ belongs to $W^{k,\hat p}(\R^2;\R^{2n})$ for all $k$ and $\hat p\in (1,+\infty)$~\cite{Shubin:91_SpectralTheoryOnNoncompMflds}.  In particular, $\xi$ is smooth and for any $s\in\R$ the function $\xi(s,\cdot)$ belongs to $W^{k,2}(\R;\R^{2n})$ for all $k$. The rest of the proof is obtained by applying similar arguments to those used in the proof of Lemma~2.11 in~\cite{Salamon:97_LecturesOnFloerHomology}. For the reader's convenience we repeat the main steps here.

Define
\begin{equation*}
  f(s)=\frac 12 \int\limits_{-\infty}^{+\infty} |\xi(s,t)|^2dt.
\end{equation*}
Then
\begin{align*}
 f''(s)&=\int\limits_{-\infty}^{+\infty}\Bigl ( |\partial_s\xi|^2 +\langle \xi, \partial^2_{ss}\xi \rangle  \Bigr )\,dt\\
     & =2\int\limits_{-\infty}^{+\infty} |\partial_s\xi|^2dt + \int\limits_{-\infty}^{+\infty}\langle \xi, (\partial_{s}S)\xi \rangle\, dt\\
    &\ge 2\int\limits_{-\infty}^{+\infty} |J_0\partial_t\xi +S\xi|^2dt -\e \int\limits_{-\infty}^{+\infty}|\xi|^2dt\\
    &\ge \delta^2  \int\limits_{-\infty}^{+\infty}|\xi|^2dt\\
    & = \delta^2 f(s).
\end{align*}
Here we have used the fact that the operator $J_0\partial_t +S(s,t)$ is invertible for $s\ge s_0$ and also the following equality:
\begin{align*}
  \int\limits_{-\infty}^{+\infty}\langle \xi, \partial_t(J_0\partial_{s}\xi) \rangle\, dt
  =&  \int\limits_{-\infty}^{+\infty}\partial_t\langle \xi, J_0\partial_{s}\xi \rangle\, dt
         -  \int\limits_{-\infty}^{+\infty}\langle \partial_t\xi, J_0\partial_{s}\xi \rangle\, dt\\
  =& 0-  \int\limits_{-\infty}^{+\infty}\langle J_0\partial_s\xi +J_0S\xi, J_0\partial_{s}\xi \rangle\, dt\\
 =& - \int\limits_{-\infty}^{+\infty} |\partial_s\xi|^2dt - \int\limits_{-\infty}^{+\infty} \langle S\xi, \partial_{s}\xi \rangle\, dt.
\end{align*}
The inequality $f''(s)\ge\delta^2 f(s)$ implies $ f(s) \le C_1e^{-\delta |s|}$. On the other hand, there exists a constant $C_2$ such that for all solutions of the equation $D\xi=0$ we have the estimate%
\begin{equation*}
  \Delta |\xi|^2\ge -C_2|\xi|^2.
\end{equation*}
This  implies the mean value inequality%
\begin{equation*}
  |\xi(s,t)|^2\le \frac {C_3}{r^2}\int_{B_r(s,t)}|\xi|^2 dsdt.
\end{equation*}
Taking into account the exponential decay of $f$ we obtain the statement of this lemma from the last inequality.
\end{proof}

\vskip0.3cm


\begin{lem}\label{Lem_EstimOnS}
 Assume $u$ is a solution of~\eqref{Eq_PseudoholomPlaneWithHamPert_mod},\eqref{Eq_BCwithoutIntegrals}. Then there exist positive constants $C_\pm$ such that the estimates
  \begin{align}
    \| S(s,t) - S_\pm(t)  \|\le 
        C_\pm\max \bigl \{ |\partial_su(s,t)|,\, d\bigl ( u(s,t),\gamma_\pm(t)   \bigr)  \bigr \}\label{Eq_EstForMatrixS}
   \end{align}
hold for all $t$ and all $s$ such that $\mp s\ge 0$. 
\end{lem}
\begin{proof}
  With the help of equations~\eqref{Eq_BrokenFlowLine} and~\eqref{Eq_PseudoholomPlaneWithHamPert_mod} we obtain
  \begin{align*}
    \bigl | \Psi(s,t)^{-1} \partial_t u(s,t)-\psi_\pm^{-1}(t)\dot\gamma_\pm(t) \bigr | %
      &\le \bigl |   \Psi(s,t)^{-1}(\partial_t u(s,t) +v^t(u(s,t)) \bigr |\\ 
          &+\bigl |   \Psi(s,t)^{-1}v^t(u(s,t)) - \psi_\pm^{-1}(t)v^t(\gamma_\pm(t))  \bigr |\\
     & \le \tilde C_\pm\max  \bigl \{ |\partial_su(s,t)|,\, d\bigl ( u(s,t),\gamma_\pm(t)   \bigr)  \bigr \}
  \end{align*}
for some  positive constants $\tilde C_\pm$ and for all  $t,s$ as in the statement of the Lemma. Estimate~\eqref{Eq_EstForMatrixS} then follows from  formulae~\eqref{Eq_MatrixS},\eqref{Eq_MatricesSpm}, and the  above inequality.
\end{proof}



\begin{thm}[Exponential decay] \label{Thm_ExpDecay}
Let $u$ be a solution of~\eqref{Eq_PseudoholomPlaneWithHamPert_mod},\eqref{Eq_BCwithoutIntegrals} with  $E_p(u)<\infty$ for some $p\in [2,\infty)$. Then the following holds: 
  \begin{itemize}
  \item[(i)]  $\partial_su\in W^{k,\hat p}(\R^2; u^*TM)$ for all $k$ and all $\hat p\in (1,\infty)$. In particular, $E_{\hat p}(u)<\infty$ for all $\hat p\in (1,\infty)$.
  \item[(ii)]  There exist positive constants $C,\delta$ such that the inequality 
\[
|\partial_su(s,t)|\le Ce^{-\delta |s|}
\]
holds for any $(s,t)\in\R^2$.
  \end{itemize}
\end{thm}
\begin{proof}
First observe that $\partial_su$ satisfies $D_u\Sigma (\partial_su)=0$ since equation~\eqref{Eq_PseudoholomPlaneWithHamPert_mod} is translation-invariant with respect to the $s$-variable. Furthermore, we claim that the operator $L$ representing $D_u\Sigma$ in the trivialization $\Psi$ is $C^\infty$--bounded. Indeed, since $u$ is $C^\infty$--bounded, so is $S(s,t)$. Obviuosly, $L$ is also uniformly elliptic and therefore statement~\textit{(i)} follows by~\cite{Shubin:91_SpectralTheoryOnNoncompMflds}.

To prove~\textit{(ii)} it is enough to prove that the matrix-valued function $S(s,t)$ defined by~\eqref{Eq_MatrixS} satisfies the hypotheses of  Lemma~\ref{Lem_ExpDecayForAnOperator}. We have already showed that $S(s,t)$ is $C^\infty$--bounded. From  Lemma~\ref{Lem_EstimOnS} and Proposition~\ref{Prop_UnifConvergenceOfDsU} we obtain that hypothesis~\textit{(ii)} of Lemma~\ref{Lem_ExpDecayForAnOperator} is satisfied.  Furthermore, by~\textit{(i)} and the Sobolev embedding theorems any solution $u$ of~\eqref{Eq_PseudoholomPlaneWithHamPert_mod} with $E_p(u)<\infty$ satisfies
\begin{equation*}
    \lim\limits_{s\to\pm\infty}\sup\limits_{t}\Bigl ( |\cd s \partial_su| +  |\cd t \partial_su| \Bigr )=0
\qquad\text{and}\qquad \sup\limits_{\mathbb R^2}|\cd t \partial_t u|<\infty.
  \end{equation*}
This implies  that hypothesis~\textit{(iii)} of Lemma~\ref{Lem_ExpDecayForAnOperator} is also satisfied. Finally, hypothesis \textit{(iv)} is satisfied, since $l_\pm$ represents $D_{\gamma_\pm}\sigma$ in the chosen trivialization.  
\end{proof}

\begin{cor}\label{Cor_FiniteEnergyImplyBCs}
  Let $u$ be a solution of~\eqref{Eq_PseudoholomPlaneWithHamPert_mod},\eqref{Eq_BCwithoutIntegrals} with  $E_p(u)<\infty$ for some $p\in [2,\infty)$. Then
\[
\lim_{t\to\pm\infty}\int_{-\infty}^{+\infty} |\partial_s u(s,t)|\, ds=0,\qquad
\lim_{s\to\pm\infty}\int_{a}^b |\partial_su(s,t)|\, dt=0,
\]
i.e., $u$ is a solution of~\eqref{Eq_PseudoholomPlaneWithHamPert_mod}--\eqref{Eq_PseudoholomPlaneWithHamPert_BC_s}.
\end{cor}
\begin{proof}
  The statement follows from the Sobolev embedding theorems as explained on p.~\pageref{Eq_AuxEstimForTrace}.  

\end{proof}

\begin{cor}\label{Cor_ApriooriWkpBounds}
  For any solution $u$ of~\eqref{Eq_PseudoholomPlaneWithHamPert_mod}--\eqref{Eq_PseudoholomPlaneWithHamPert_BC_s} $\partial_su\in W^{k,p}(\R^2; u^*TM)$ for all $k$ and all $p\in [1,\infty)$. Moreover, for each $k\ge 0$ and $p\ge 1$ there exists a constant $C_{k,p}$ independent of $u$ such that
\[
\| \partial_su \|_{W^{k,p}}\le C_{k,p}.
\]
\end{cor}
\begin{proof}
First observe that by Theorem~\ref{Thm_EnergyId}  $\partial_su$ belongs to $L^2(\R^2; u^*TM)$ and $\|\partial_su\|_{L^2}$  is bounded by a constant independent of $u$.

  As already mentioned in the proof of Theorem~\ref{Thm_ExpDecay} the matrix-valued function $S(s,t)$ is $C^\infty$--bounded. Moreover, it follows from Corollary~\ref{Cor_AprioriEstimForCkNorms}  that  the corresponding bounds can be chosen to be independent of $u$. Furthermore, the operator $L$ is uniformly elliptic with the corresponding constant also independent of $u$. For such an operator of order $1$  we have the a priori estimate
\[
\| \xi \|_{W^{k,2}}\le C_{k}\bigl (  \| L\xi\|_{W^{k-1,2}} +\| \xi\|_{L^2}  \bigr),
\qquad k\ge 1,
\]
where the constant $C_k$ does not depend on  $u$ (this is seen by examining explicit formulae for a parametrix of $L$). This implies the statement of the corollary for all $k\ge 1$ and $p=2$. This special case implies in turn the statement of the corollary in general by the Sobolev embedding theorems.  
\end{proof}


\subsection{Compactness}

\begin{prop}\label{Prop_FinEnergyImpliesLimits}
  Let $u$ be a solution of~\eqref{Eq_PseudoholomPlaneWithHamPert_mod} and~\eqref{Eq_PseudoholomPlaneWithHamPert_BC_t} with $E_p(u)<\infty$ for some $p\in [2,\infty)$. Then there exist solutions $\gamma_\pm$ of problem~\eqref{Eq_BrokenFlowLine} such that
  \begin{equation*}
    \lim\limits_{s\to\pm\infty} u(s,t)= \gamma_\mp(t)\qquad\text{and}\qquad \lim_{s\to\pm\infty}\int_{a}^b |\partial_su(s,t)|\, dt=0,
  \end{equation*}
where the limits on the left hand side are understood in the $C^0(\R)$--topology and $a\le b$ are arbitrary.
\end{prop}

\begin{proof}
  
 The proof consists of the following three steps.
\setcounter{thestep}{0}

 \begin{step}\label{Step_ConvergentSubseq}
   Let $\b_n\in C^1(\R;  M)$ be an arbitrary sequence of curves such that the following holds:
   \begin{itemize}
   \item[(i)] There exists a compact subset $\hat K\subset M$ containing the images of all curves $\beta_n$;
   \item[(ii)] $\lim\limits_{n\to\infty}\sup\limits_{t\in\mathbb R}\,\bigl | \dot\b_n +v^t  \bigr |=0$;
   \item[(iii)] $\b_n(t)\to m_\mp\quad\text{as}\ t\to\pm\infty\qquad\text{uniformly with respect to}\ n$. 
   \end{itemize}
Then there exists a subsequence $\b_{n_k}$ converging in $C^0(\mathbb R; M)$ to a solution of~\eqref{Eq_BrokenFlowLine}.
 \end{step}

Recall that the function $\rho=|v_0|^2=|v^t|^2$ is bounded. Hence, it follows from~\textit{(ii)} that the sequence $\beta_n$ is equicontinuous. By the Arzela--Ascoli theorem there exists a subsequence  $\b_{n_k}$ convergent on any finite interval to some $\gamma\in C^0(\mathbb R; M)$.  Then $\gamma\in C^1(\mathbb R; M)$ and $\dot\gamma+ v^t=0$.

  Furthermore, by~\textit{(iii)} for any $\e>0$ there exits $T_\e>0$ such that for all $t\ge T_\e$ and all $n_k$ we have $d(\b_{n_k}(t), m_-)\le\e$. Then $d(\gamma(t), m_-)=\lim_{k\to\infty}  d(\b_{n_k}(t), m_-)\le\e$ provided $t\ge T_\e$. Hence $\lim_{t\to +\infty}\gamma(t)=m_-$ and similarly   $\lim_{t\to -\infty}\gamma(t)=m_+$, i.e., $\gamma$ is a solution of~\eqref{Eq_BrokenFlowLine}. Then 
 \begin{equation*}
   \sup\limits_{t\in\mathbb R}d(\b_{n_k}(t),\gamma(t))\le
       \max\Bigl\{ \sup\limits_{t\in [-T_\e, T_\e]}d(\b_{n_k}(t),\gamma(t)),\ 2\e\Bigr\}\le 2\e
 \end{equation*}
provided $n_k$ is large enough. This finishes the proof of Step~\ref{Step_ConvergentSubseq}.

\begin{step}\label{Step_DistanceToAntigradFlowLines}
Let $u$ be a solution of~\eqref{Eq_PseudoholomPlaneWithHamPert_mod} and~\eqref{Eq_PseudoholomPlaneWithHamPert_BC_t} with $E_p(u)<\infty$ for some $p\in [2,\infty)$. Then for any $\e>0$ there exists $\sigma_\e>0$ such that
  \begin{equation*}
    \inf\limits_{\gamma\in\Gamma(m_-;m_+)}\  \sup\limits_{t\in\mathbb R}
d\bigl (u(s,t), \gamma(t)\bigr )\le\e\qquad\text{provided}\ |s|\ge\sigma_\e.
  \end{equation*}
\end{step}

Assume the converse. Then there exists a sequence $s_n\to +\infty$ such that for $\b_n(t)=u(s_n,t)$ we have  
\begin{equation}\label{Eq_AuxIneqA1}
  \sup\limits_{t\in\mathbb R}d\bigl (\b_n(t), \gamma(t)\bigr )\ge\e_0\qquad\text{for all }\ \gamma\in\Gamma (m_-;m_+).
\end{equation}
By Theorem~\ref{Thm_AprioriC0Estimates} and Proposition~\ref{Prop_UnifConvergenceOfDsU} the sequence $\b_n$ satisfies the hypothesis of Step~\ref{Step_ConvergentSubseq} and hence has a convergent subsequence. But this contradicts inequality~\eqref{Eq_AuxIneqA1}.

\begin{step}
  We prove the proposition.
\end{step}

Let $u$ satisfy the hypotheses of the Proposition. Since $\Gamma (m_-;m_+)$ is discrete,  by  Step~\ref{Step_DistanceToAntigradFlowLines} the family $u(s,\cdot)$ converges to some $\gamma_\pm\in\Gamma(m_-;m_+)$ in $C^0(\R)$ as $s\to\mp\infty$. The rest follows immediately from Proposition~\ref{Prop_UnifConvergenceOfDsU}.
\end{proof}

\begin{prop}\label{Prop_AprioriBoundsForTLarge}
  For any $\e>0$ there exists $T>0$ such that for all solutions $u$ of~\eqref{Eq_PseudoholomPlaneWithHamPert_mod}--\eqref{Eq_PseudoholomPlaneWithHamPert_BC_s} the following holds:
\begin{align*}
  (i) & \int\limits_{\R\times [T,+\infty)} |\partial_su|^2\, dsdt <\e; & \qquad\qquad
(ii) & \int\limits_\R |\partial_s u(s,t)|\, ds<\e\  \text{ for all } t\ge T;\\
 (iii) & \int\limits_{\R\times [-\infty, -T]} |\partial_su|^2\, dsdt <\e; & \qquad\qquad
(iv) & \int\limits_\R |\partial_s u(s,t)|\, ds<\e\  \text{ for all } t\le -T.
\end{align*}
\end{prop}
\begin{proof}

 By Corollary~\ref{Cor_ApriooriWkpBounds} we have the inequality
\[
\int_{\R^2}|\partial_su(s,t)|\, dsdt< C_{0,1}.
\]      
This implies that for any $\e>0$ and any $T>0$ there exists $\tau\in [T, T+\e^{-1}C_{0,1}]$ such that the estimate holds:
\begin{equation}\label{Eq_EstOnTheLength}
\int_{\R}|\partial_su(s,\tau)|\, ds< \e.
\end{equation}      

Arguing like in the proof of the  energy identity we obtain the equality 
\begin{align*}
  \int\limits_{\R\times [\tau, +\infty)} |\partial_su|^2\, ds\,dt &= I(\tau) - \int\limits_\R\lambda\bigl (\partial_s u(s,\tau)\bigr )\, ds,\\
I(\tau)&=\int\limits_\tau^{+\infty} \Bigl (
 \l (\dot\gamma_+) -  \l (\dot\gamma_-) +
  \im e^{-i\theta(t)}\bigl (f\comp\gamma_+ -f\comp\gamma_-    \bigr ) 
\Bigl )\, dt.
\end{align*}

Pick any $\e>0$ and choose $T_0>0$ so large that $|I(\tau)|<\e$ for all $\tau\ge T_0$. Then, as we have shown above, there exists $\tau\in [T_0, T]$ such that estimate~\eqref{Eq_EstOnTheLength} holds, where $T=T_0+\e^{-1}C_{0,1}$. Hence, we obtain
\[
\int\limits_{\R\times [T,+\infty)} |\partial_su|^2\, dsdt \le \int\limits_{\R\times [\tau,+\infty)} |\partial_su|^2\, dsdt\le |I(\tau)| + \Lambda\int_\R |\partial_s(s,\tau)|\, ds\le \e +\Lambda\e,
\]
where the constant $\Lambda$ depends on $\l$ only. This proves estimate \textit{(i)}.

Let us prove  \textit{(ii)}. We choose $T>0$ so that \textit{(i)} holds. Arguing like in the proof of Corollary~\ref{Cor_ApriooriWkpBounds} for $T'>T$ we obtain  the inequality
\[
\| \partial_su \|_{W^{k,p}(\R\times [T',+\infty))}\le \tilde C_{k,p}\,\e,
\] 
where the constant  $\tilde C_{k,p}$ does not depend on $u$. This in turn implies that there exists a constant $\tilde C$ independent of $u$ such that  the inequality   $\int_\R |\partial_s u(s,t)|\, ds<\tilde C\e$ holds for  all  $t\ge T$. This finishes the proof of~\textit{(ii)}. 

The remaining inequalities are proved in a similar manner.   
\end{proof}

\begin{thm}
  Let $u_k\in\mathcal M(\gamma_-;\gamma_+)$ be any sequence. Then there exists a subsequence (still denoted by $u_k$) and subsequences $s_k^j,\ j=1,\dots, l$, such that $u_k(s+ s_k^j, t)$ converges with its derivatives uniformly on compact subsets of $\R^2$ to $u^j\in\mathcal M(\gamma^{j-1};\gamma^{j})$, where $\gamma^0=\gamma_-,\ \gamma^l=\gamma_+$. 
\end{thm}
\begin{proof}
Denote 
\[
d_0=\frac 13\inf\Bigl\{   
d\bigl ( \gamma(0),\delta (0) \bigr )\mid \gamma,\delta\in\mathcal M(m_-; m_+),\ \gamma\neq\delta  
\Bigr\}.
\]  
For an arbitrary sequence  $u_k\in\mathcal M(\gamma_-;\gamma_+)$ put
\[
s_k^1=\sup \{ s\in\R\mid d\bigl ( u_k(s,0), \gamma_-(0)  \bigr )>d_0\;\}.
\]
Notice that by the definition of $s_k^1$ we have 
\begin{equation}\label{Eq_Aux_DistanceToLimitingCurve}
  d\bigl ( u_k(s_k^1,0),\gamma_-(0) \bigr ) = d_0\qquad\text{and}\qquad 
 d\bigl ( u_k(s+s_k^1,0),\gamma_-(0) \bigr ) \le d_0\quad\text{for all } s\ge 0.
\end{equation}

Since the sequence $\sup_{\R^2}\{ |\partial_su_k|,\ |\partial_tu_k|\}$ is bounded, by~\cite[Lemma~5.2]{Salamon:90MorseTheory} we obtain that the sequence $u_k(s+s_k^1, t)$ has a subsequence (still denoted by the same letter) uniformly converging with its derivatives to a map $u^1\colon \R^2\to M$ on  compact subsets of $\R^2$. Clearly, $u^1$ is a solution of~\eqref{Eq_PseudoholomPlaneWithHamPert_mod} with $E_2(u^1)\le \mathscr F(\gamma_+)-\mathscr F(\gamma_-)$. Moreover, by Proposition~\ref{Prop_AprioriBoundsForTLarge}~\textit{(ii)} for any $\e>0$ there exists $T>0$ such that for any  $a,b\in\R,\ a<b$ we have 
\[
\int\limits_a^b|\partial_su^1(s,t)|\, ds=\lim_{k\to\infty}\int\limits_a^b |\partial_su_k(s+s_k^1,t)|\, ds\le\e\quad\Longrightarrow\quad \int\limits_{-\infty}^{+\infty} |\partial_su^1(s,t)|\, ds\le\e
\]
provided $t>T$. This implies that condition~\eqref{Eq_PseudoholomPlaneWithHamPert_BC_t} holds for $u=u^1$. Then, by  Proposition~\ref{Prop_FinEnergyImpliesLimits} we obtain that there exist $\gamma^0,\gamma^1\in\mathcal M(m_-;m_+)$ such that 
\begin{align*}
 & \lim_{s\to +\infty} u^1(s,t)=\gamma^0(t),  & &  \lim_{s\to +\infty}\int_a^b |\partial_s u^1(s,t)|\, dt =0,\\
 & \lim_{s\to -\infty} u^1(s,t)=\gamma^1(t),  & &  \lim_{s\to -\infty}\int_a^b |\partial_s u^1(s,t)|\, dt =0.
\end{align*}
On the other hand,  from~\eqref{Eq_Aux_DistanceToLimitingCurve} we obtain that $d\bigl ( u^1(s,0), \gamma_-(0) \bigr )\le d_0$ for all $s\ge 0$. Hence $\gamma^0=\gamma_-$.

We are done if $\gamma^1=\gamma_+$. If this is not the case we proceed by induction. Having established the existence of the sequences $s_k^j$ such that $u_k(s+s_k^j, t)$ converges to $u^j\in\mathcal M(\gamma^{j-1},\gamma^j)$ for $j=1,\dots, q$ we choose $s^*<0$ such that  $d\bigl ( u^q(s^*,0), \gamma^q(0) \bigr ) < d_0$. For $k$ sufficiently large we then have $d\bigl ( u_k(s_k^q+s^*,0), \gamma^q(0) \bigr ) < d_0$. Define
\[
s_k^{q+1}=\inf\bigl\{
s\le s_k^q +s^*\mid\quad d\bigl ( u_k(\sigma,0), \gamma^q(0) \bigr )\le d_0\quad\text{for } s\le\sigma\le s_k^q+s^*
  \bigr\}.
\]
Passing to a subsequence if necessary we may assume that $u_k(s+s_k^{q+1},t)$ converges to $u^{q+1}\in\mathcal M(\gamma^q,\gamma^{q+1})$ with $\gamma^{q+1}\neq \gamma^{q}$. This finishes the induction step. Finally, the process is finite, since for all $q=1,\dots, l$ we must have $\mathscr F(\gamma^{q-1})<\mathscr F(\gamma^{q})$.
\end{proof}

\begin{cor}
  The space 
\[
\check{\mathcal M}(m_-; m_+)=\bigcup_{\gamma_\pm\in\mathcal M(m_-;m_+)} \mathcal M(\gamma_-;\gamma_+) 
\]
is compact.
\end{cor}


\subsection{Fredholm property}

The following result, which simplifies the arguments used in earlier versions of the preprint, has been communicated to the author by V.~Rabinovich.

\begin{lem}\label{Lem_UniformlyEllipticOpAreFredholm}
Let 
\[
A=\sum_{|\alpha|\le m}a_\a(x)\frac \partial{\partial x^\a},\qquad x\in\R^n
\]
be a uniformly elliptic $C^\infty$--bounded differential operator of order $m$, where $a_\a$ takes values in the space of  $l\times l$-matrices. If  $A\colon W^{m,2}(\R^n;\R^l)\rightarrow L^2(\R^n;\R^l)$ is Fredholm, then  $A\colon W^{k+m,p}(\R^n;\R^l)\rightarrow W^{k,p}(\R^n;\R^l)$ is Fredholm for all  $k\in\R$, $p>1$ and its index depends neither on $k$ nor on $p$.
\end{lem}
\begin{proof}
Following~\cite{RabinovichRoch:08_AgmonsEstimates} we say that $A^g=\sum_{|\alpha|\le m}a_\a^g(x)\frac \partial{\partial x^\a}$ is a limit operator of $A$ if for some sequence $g\colon\mathbb N\to\R^n$ such that $g_j\to\infty$ we have $a_\a(x+g_j)\to a_\a^g(x)$ uniformly on all compact subsets of $\R^n$. Then by Theorem 5.6 of~\cite{RabinovichRoch04_WienerAlgebras} (see also Theorem~2 of~\cite{RabinovichRoch:08_AgmonsEstimates}) the operators $A^g\colon W^{m,2}(\R^n;\R^l)\to L^2(\R^n;\R^l)$ are invertible. Observe that $A^g$ is a (pseudo)differential operator with the symbol from H\"ormander's class $S^m_{1,0}$. By~\cite[Theorem~3.2]{Beals:77_CharOfPseudodiffOp} the inverse $(A^g)^{-1}$ is a pseudodifferential operator with the symbol from $S^{-m}_{1,0}$. Hence, $(A^g)^{-1}\colon W^{k-m,2}(\R^n;\R^l)\to W^{k,2}(\R^n;\R^l)$ is bounded for any $k\in\R$. Applying~\cite[Theorem 5.6]{RabinovichRoch04_WienerAlgebras} again we obtain that $A\colon W^{k+m,2}(\R^n;\R^l)\rightarrow W^{k,2}(\R^n;\R^l)$ is Fredholm for any $k\in\R$.

  For arbitrary $k$ and $p>1$  put $k'=\min\{k-1, k-\frac np+\frac n2 \}$ to obtain the embeddings  $W^{k,p}(\R^n;\R^l)\hookrightarrow W^{k',2}(\R^n;\R^l), W^{k+m,p}(\R^n;\R^l)\hookrightarrow W^{k'+m,2}(\R^n;\R^l)$.  We claim that $A\colon W^{k+m,p}(\R^n;\R^l)\to W^{k,p}(\R^n;\R^l)$ has a closed range. Indeed, let $\zeta_n$ be any sequence from $A\bigl ( W^{k+m,p}(\R^n;\R^l) \bigr )$ converging to $\zeta_0$ in $W^{k,p}(\R^n;\R^l)$. Then $\zeta_n\in A\bigl ( W^{k'+m,2}(\R^n;\R^l) \bigr )$ converges to $\zeta_0$ in $W^{k',2}(\R^n;\R^l)$. Hence  $\zeta_0=A\xi_0$ for some $\xi_0\in W^{k'+m,2}(\R^n;\R^l)$.  Since $A$ is $C^\infty$--bounded uniformly elliptic operator,  $\zeta_0\in W^{k,p}(\R^n;\R^l)$ implies that $\xi_0\in W^{k+m,p}(\R^n;\R^l)$. This proves that $A\bigl ( W^{k+m,p}(\R^n;\R^l) \bigr )$ is closed in $W^{k,p}(\R^n;\R^l)$.

Furthermore,  $C^\infty$--boundedness and  uniform ellipticity imply that if $\xi\in W^{k,p}(\R^n;\R^l)$ is in the kernel of $A$ for some $k$ and $p$, then $\xi\in W^{k,p}(\R^n;\R^l) $ for all $k$ and $p$. In particular,  for any $k$ and $p$ the dimension of $\ker\bigl (A\colon W^{k+m,p}(\R^n;\R^l) \to W^{k,p}(\R^n;\R^l) \bigr )$ is finite and depends neither on $k$ nor on $p$. Moreover, applying similar arguments to the formal adjoint operator of $A$ we obtain that  the dimension of $\coker\bigl (A\colon W^{k+m,p}(\R^n;\R^l) \to W^{k,p}(\R^n;\R^l) \bigr )$ is also finite and depends neither on $k$ nor on $p$.
\end{proof}


In the lemma below we use the same notations as in Lemma~\ref{Lem_ExpDecayForAnOperator}. 

\setcounter{thestep}{0}
\begin{lem}
  Assume that in addition to hypotheses \textit{(i)}--\textit{(iii)} of Lemma~\ref{Lem_ExpDecayForAnOperator} the following holds:
  \begin{itemize}
  \item[(a)] For each $t\in\R$ the matrix $J_0S_\pm(t)$ is symmetric;
  \item[(b)] $S(s,t)$ converges to constant matrices $H_\pm$ in the  $C^0$--topology as $t\to\mp\infty$. Moreover, 
\[
H_+=\lim_{t\to -\infty}S_+(t)=\lim_{t\to -\infty}S_-(t),\qquad
H_-=\lim_{t\to +\infty}S_+(t)=\lim_{t\to +\infty}S_-(t) 
\] 
are symmetric matrices.
  \end{itemize}
Then $L\colon W^{k+1,p}(\R^2;\R^{2n})\to W^{k,p}(\R^2;\R^{2n})$ is Fredholm for any $k\in\R$, $p>1$ and its index depends neither on $k$ nor on $p$.
\end{lem}
\begin{proof}

The proof consists of the following three steps.

\begin{step}\label{Step_TranslInvarOperatorsAreIso}
Consider the $s$--independent operators
\[
L_\pm =\partial_s +J_0\partial_t +S_\pm(t).
\]
Then $L_\pm\colon W^{1,2}(\R^2;\R^{2n})\to L^2(\R^2;\R^{2n})$ are invertible.     
\end{step}

Since $l_\pm\colon W^{1,2}(\R;\R^{2n})\to L^2(\R;\R^{2n})$ are isomorphisms we have the estimates
\begin{equation}\label{Eq_AuxiliaryEstimateForLpm}
\|\eta\|_{W^{1,2}}\le C_\pm \| l_\pm\eta\|_{L^2},\qquad \text{for all }\eta\in  W^{1,2}(\R;\R^{2n}).
\end{equation}
Pick any $\zeta\in C^\infty(\mathbb R^2;\R^{2n})$ with compact support  and apply the Fourier transform in the variable $s$ to the equation $L_\pm(\eta)=\zeta$ to obtain%
\begin{equation*}
 i\sigma\,\hat\eta(\sigma, t) +J_0l_\pm\hat\eta(\sigma, t)=\hat\zeta(\sigma ,t).
\end{equation*}
Observe that $J_0l_\pm$ is a  symmetric operator such that $0$ does not belong to the spectrum of $J_0l_\pm$. Hence, the above equation is solvable for any real $\sigma$.  Applying the inverse Fourier transform we obtain a solution $\eta$ of the initial equation $L_\pm(\eta)=\zeta$. Moreover, with the help of~\eqref{Eq_AuxiliaryEstimateForLpm}   an easy computation yields the estimate $\| \eta\|_{W^{1,2}}\le \tilde C_\pm\| \zeta\|_{L_{2}}$.  This implies that $L_\pm$ are isomorphisms.

\begin{step}
  Consider the operators 
\[
K_\pm =\partial_s + J_0\partial_t + H_\pm
\]
with constant coefficients. Then $K_\pm\colon W^{1,2}(\R^2;\R^{2n})\to L^2(\R^2;\R^{2n})$ are invertible. 
\end{step}

Write $J_0K_\pm = -\partial_t + J_0\partial_s + J_0H_\pm$ and observe that the operators 
\[
k_\pm=\tfrac d{ds} +H_\pm\colon W^{1,2}(\R;\R^{2n})\to L^2(\R;\R^{2n})
\] 
are isomorphisms. Indeed,  any function satisfying $k_\pm\eta=0$ can be expressed through exponential functions and therefore does not belong to $W^{1,2}(\R; \R^{2n})$. Similarly, the cokernel of $k_\pm$ is also trivial. The rest of the proof of this step is analogous to the proof of Step~\ref{Step_TranslInvarOperatorsAreIso}.


\begin{step}\label{Step_EndOfLemmaUniformEllipticAreFredholm}
  We prove the lemma.
\end{step}


Clearly, any limit operator $L_0=\partial_s +J_0\partial_t +S_0(s,t)$ of $L$  must be $K_\pm$ or
\[
L_\pm^\tau=\partial_s +J_0\partial_t +S_\pm(t+\tau)= V_\tau L_\pm V_{-\tau},
\]
where $V_\tau$ denotes the shift operator $\xi(t)\mapsto \xi(t+\tau)$.  Since $V_\tau$ acts as an isomorphism on $W^{k,2}(\R^2, \R^{2n})$ for any $k$, the operator $L_\pm^\tau\colon W^{1,2}(\R^2, \R^{2n})\to L^2(\R^2, \R^{2n})$ is also an isomorphism. Hence, by~\cite[Theorem~5.6]{RabinovichRoch04_WienerAlgebras}  we obtain that $L\colon W^{1,2}(\R^2, \R^{2n})\to L^2(\R^2, \R^{2n})$ is Fredholm. Then the statement of the lemma follows  from Lemma~\ref{Lem_UniformlyEllipticOpAreFredholm}.
\end{proof}


\begin{thm}\label{Thm_LinearizationIsFredholm}
For each solution $u$ of~\eqref{Eq_PseudoholomPlaneWithHamPert_mod}--\eqref{Eq_PseudoholomPlaneWithHamPert_BC_s} the  map  $D_u\Sigma\colon W^{k+1,p}(\mathbb R^2; u^*TM)\rightarrow W^{k,p}(\mathbb R^2; u^*TM)$ is Fredholm for any $k\in\R$, $p> 1$ and its index depends neither on $k$ nor on $p$.
 \end{thm}
\begin{proof}
Clearly, it is enough to check that the matrix--valued function $S(s,t)$ given by~\eqref{Eq_MatrixS} satisfies the hypotheses of Lemma~\ref{Lem_UniformlyEllipticOpAreFredholm}. The fact that \textit{(a)} holds can be checked by direct computation using~\eqref{Eq_MatricesSpm} and is well known~\cite{Salamon:97_LecturesOnFloerHomology}. To see that \textit{(b)} holds, observe that $v^t=v_0$ for $|t|$ large enough. It follows that $H_\pm$ represents $J\nabla v_0=\nabla v_1$ at $m_\pm$. Here we used the fact, that $J$ is integrable in a neighbourhood of $m_\pm$. It remains to notice that $\nabla v_1$ is the Hessian of $f_1=\im f$ at $m_\pm$  and therefore is symmetric. 
\end{proof}



To compute the index of $D_u\Sigma$ we need some preparation. Since $\Ind{(D_u\Sigma)}$ depends neither on $k$ nor on $p$, we can put  $k=1,\ p=2$.  With an arbitrary $C^1$--curve $\gamma\colon\mathbb R\rightarrow M$ satisfying %
 \begin{equation}\label{Eq_AuxAsymptoticConditionsForCurves}
   \lim\limits_{t\to\pm\infty}\gamma (t)=m_\mp\quad\text{and }\  \lim\limits_{t\to\pm\infty}\dot\gamma (t)=0
 \end{equation}
we associate a pair of Lagrangian subspaces in $T_{\gamma(0)}M$ as follows. Consider the operator%
\begin{equation*}
   A\colon C^\infty (\gamma^*TM)\rightarrow C^\infty (\gamma^*TM),\qquad A\xi =J\cd t\xi + \tilde S\xi,
\end{equation*}
where $\tilde S$ is a zero-order operator, namely $\tilde S\xi = \cd \xi\tilde v^t + (\nabla_\xi J)\dot\gamma, \ \tilde v^t= \cos\theta_\nu(t)v_1 -\sin\theta_\nu(t) v_0$  (compare with~\eqref{Eq_LinerizationOfHolPlane}). Notice that $\lim\limits_{t\to\mp\infty} \tilde S = \tilde S^\pm\in End(T_{m_\pm}M)$  
is the Hessian of $\im f$ at $m_\pm$. As we already observed in the proof of Theorem~\ref{Thm_LinearizationIsFredholm} $J\tilde S^\pm$ is then the Hessian of $\re f$ and therefore is a non-degenerate  self-adjoint endomorphism with vanishing signature (i.e., $J\tilde S^\pm$ has $n$ positive and $n$ negative eigenvalues).


Denote by $\xi_{\rm v},\ \rm v\in T_{\gamma(0)}M$, a solution of the Cauchy problem 
$A\xi_{\rm v}=0,\ \xi_{\rm v}(0)=\rm v$ and put 
\begin{equation*}
  \Lambda^\pm =\{ \, \rm v\in  T_{\gamma(0)}M\ \; |\;  \lim\limits_{t\to\mp\infty} \xi_{\rm v}(t) =0 \}. 
\end{equation*}
Then $\Lambda^\pm$ are Lagrangian subspaces. Indeed, a straightforward computation shows that $\om (\xi_{\rm v}(t), \xi_{\rm w}(t))$ does not depend on $t$ for any $\rm{v,w}\in T_{\gamma(0)}M$. Therefore, if $\rm{v,w}\in\Lambda^+$, then $\om(\rm v, \rm w)=0$ since $\om (\xi_{\rm v}(t), \xi_{\rm w}(t))$ vanishes at $-\infty$. Besides, $\dim\Lambda^+=n$ since  the signature of $J\tilde S^+$ vanishes.

\begin{rem}\label{Rem_CrossingsCoincide}
 If $\rm v\in\Lambda^{\pm}$, then $\xi_{\rm v}$ decays exponentially fast at $\mp\infty$ since $J\hat S^\pm$ is nondegenerate and self-adjoint. Hence, the kernel of the operator $A\colon W^{1,2}(\gamma^* TM)\rightarrow L^2(\gamma^*TM)$ can be identified with $\Lambda^+\cap\Lambda^-$. In particular, $\ker A$  is nontrivial if and only if $\Lambda^+\cap\Lambda^-\not =\{ 0 \}$.
\end{rem}


  Further, pick any two curves $\gamma_\pm$ satisfying~(\ref{Eq_AuxAsymptoticConditionsForCurves}) such that the associated pairs of Lagrangian subspaces are transverse.  Let $u\colon\mathbb R^2\rightarrow M$ be any $C^1$--map such that each curve $\gamma_s(t)= u(s,t)$ also satisfies~(\ref{Eq_AuxAsymptoticConditionsForCurves}) and $\gamma_s\to\gamma_\pm$ as $s\to\mp\infty$ in the $C^1$-topology. With the help of the relative Maslov index for Lagrangian pairs~\cite{RobbinSalamon:93_MaslovIndexForPaths} we associate with the triple $(\gamma^+,\gamma^-; u)$ an integer $\mu(\gamma^+,\gamma^-; u)$, which is referred to as  the \emph{relative Maslov index}. To define $\mu(\gamma^+,\gamma^-; u)$ denote by $\mathcal L (TM)$ the Lagrangian Grassmannian bundle and put $\beta_0 (s)=\gamma_s(0)=u(s,0)$. Then we obtain a pair of sections $(\Lambda^+, \Lambda^-)$ of the bundle $\beta_0^*\mathcal L(TM)$ such that the subspaces $\Lambda^+(s)$ and $\Lambda^-(s)$ are transverse for $s=\pm\infty$.  Choose a unitary trivialization of $\beta_0^*TM$ and represent $\Lambda^\pm$ by a pair of curves $\Lambda^\pm_0\colon\mathbb R\rightarrow\mathcal L(\mathbb R^{2n})$. It is said that a crossing, i.e., a point $s_0$ such that $\Lambda^+(s_0)\cap\Lambda^-(s_0)\not =0$,  is \emph{regular} if  the associated crossing form~\cite{RobbinSalamon:93_MaslovIndexForPaths} $\Gamma(\Lambda^+,\Lambda^-, s_0)\colon \Lambda^+(s_0)\cap\Lambda^-(s_0)\rightarrow \mathbb R$ is nondegenerate. If all crossings are regular, then the number%
\begin{equation*}
 \mu(\gamma^+,\gamma^-; u)=\mu(\Lambda^+_0, \Lambda^-_0)=%
\sum\limits_{s_0\ \text{is crossing}} \sign\Gamma(\Lambda^+,\Lambda^-, s_0)\ \in\mathbb Z
\end{equation*}
does not depend on the choice of the unitary trivialization, i.e. the relative Maslov index is well-defined.

\begin{prop}
 With the same notations as in Theorem~\ref{Thm_LinearizationIsFredholm}, the index of $D_u\Sigma$ is given by
 \begin{equation*}
  \Ind(D_u\Sigma)=\mu(\gamma^+,\gamma^-; u).
 \end{equation*}
\end{prop}
\begin{proof}
 We follow the line of argument in~\cite{Salamon:97_LecturesOnFloerHomology}.
 
 Choose a $C^1$--small  perturbation $\hat u$ of the map $u$ with the following properties: $D_{\hat u}\Sigma$ is Fredholm, %
 $\Ind (D_{\hat u}\Sigma)=\Ind (D_u\Sigma)$, the Lagrangian pairs associated with the curves $\hat u(\pm\infty, t)$ are transverse, and there exists $T>0$ such that $\hat u(s,\pm t)=m_\mp$ for all $t\ge T$. Construct also a unitary trivialization of $\hat u^*TM$ as described in Remark~\ref{Rem_UnitaryTrivialization}. Write $D_{\hat u}\Sigma$ in the form %
 \begin{equation*}
   L(\xi)=\partial_s\xi+ A(s)\xi,\qquad \xi\colon\mathbb R^2\rightarrow \mathbb R^{2n},
 \end{equation*}
 where $A(s)\xi= J_0\partial_t\xi +\hat S(s,t)\xi$. Since the limits of the matrix $S(s,t)$ associated with $u$ are symmetric, up to a compact perturbation we can also assume that matrix $\hat S(s,t)$ is symmetric for all $(s,t)$. By the choice of $\hat u$ we also  have $\hat S(s,\pm t)=H_\mp$ for $t\ge T$, where $H_\mp$ represents the Hessian of $\im f$ at $m_\mp$.

Further, denote $\mu_0=\min\{ |\mu|\, :\, \ker (J_0H_\pm-\mu)\not =0\, \}>0$ and consider $A(s)$ as an unbounded operator in $L^2(\mathbb R;\mathbb R^{2n})$ with the domain $W^{1,2}(\mathbb R;\mathbb R^{2n})$. Then for all $s\in\mathbb R$ any point of the spectrum of $A(s)$ from the interval $(-\mu_0, \mu_0)$ is an eigenvalue. Indeed, for any $(s,\mu)\in \mathbb R\times (-\mu_0, \mu_0)$ the operator %
$A(s)-\mu\colon W^{1,2}(\mathbb R;\mathbb R^{2n})\rightarrow L^2(\mathbb R;\mathbb R^{2n})$ is Fredholm, since $J_0H_{\pm}-\mu$ is nondegenerate. Moreover, $\Ind(A(s)-\mu)=\Ind A(s)=\Ind(-J_0A(s))=0$ since the indices of $-J_0A(+\infty)$ and $D_{\gamma_-}\sigma$ coincide. Hence, if $\mu$ is not an eigenvalue of $A(s)$, then $A(s)-\mu\colon W^{1,2}(\mathbb R;\mathbb R^{2n})\rightarrow L^2(\mathbb R;\mathbb R^{2n})$ is bijective and therefore $\mu$ belongs to the resolvent set of $A(s)$.

From the above observation follows~\cite{AtiyahPatodiSinger:76_SpectralAsymmetryIII} that the index of $L$ can be computed with the help of the spectral flow of $A(s)$. Namely, a point $s_0$ is said to be a regular crossing of the family $A(s)$ if $\ker A(s_0)\not =0$ and the crossing form
\begin{equation*}
 \Gamma(A, s_0)\xi=\langle \xi, (\partial_s A) \xi \rangle_{L^2}=%
    \langle \xi, \partial_s S(s_0,\cdot) \xi \rangle_{L^2},\qquad \xi\in\ker A(s_0)
\end{equation*}
is nondegenerate. Then, if $A(s)$ has only regular intersection points, we have: $\Ind L=\sum_{s_0}\sign\Gamma (A,s_0)$, where the summation runs over all crossings $s_0$.

It follows from Remark~\ref{Rem_CrossingsCoincide} that crossings of $(\Lambda^+_0,\Lambda^-_0)$ and  $A(\cdot)$ coincide. Therefore to complete the proof it suffices to show that under the natural identification $\Lambda^+_0(s_0)\cap\Lambda^-_0(s_0)\cong\ker A(s_0)$ the associated crossing forms coincide at each crossing $s_0$ (we can assume that only regular crossings occur). 

Let $\Xi(s,t)$ be the solution operator of $A(s)$, i.e., $\Xi(s,t)$ is a square matrix of dimension $2n$ satisfying
\begin{equation*}
 J_0\partial_t \Xi +S(s,t)\Xi=0,\quad \Xi(s,0)=\mathbbm 1.
\end{equation*}
Since $S(s,t)$ is symmetric, $\Xi(s,t)\in Sp(2n;\mathbb R)$ for all $(s,t)$. From the equality%
\begin{equation*}
 \partial_t(\Xi^TJ_0\partial_s\Xi)=-(\Xi^TSJ_0) J_0\partial_s\Xi +\Xi^TJ_0\partial_s(J_0S\Xi)=%
  - \Xi^T\partial_sS\,\Xi
\end{equation*}
we obtain $ \bigl \langle \Xi\xi_0,\, \partial_sS\Xi \xi_0 \bigr \rangle =  -\partial_t\ \bigl \langle \Xi\xi_0,\, J_0\partial_s\Xi\xi_0 \bigr \rangle = %
 \partial_t\,\om_0\bigl (\Xi\xi_0,\,\partial_s\Xi\xi_0\bigr )$, where $\xi_0\in\mathbb R^{2n}$. 
 Hence, for any crossing $s_0$ and any $\xi_0\in \Lambda^+_0(s_0)\cap\Lambda^-_0(s_0)$ we have:%
\begin{equation}\label{Eq_CrossingFormForFamily}
 \begin{aligned}
   \Gamma(A, s_0)\xi_0 &=%
      \int_{-\infty}^{+\infty} \bigl \langle \Xi(s_0,t)\xi_0,\, \partial_sS(s_0, t)\Xi(s_0,t)\xi_0 \bigr \rangle\, dt\\%
      &=\lim\limits_{t\to +\infty} \om_0\bigl (\Xi(s_0,t)\xi_0,\,\partial_s\Xi(s_0, t)\xi_0\bigr ) - 
           \lim\limits_{t\to -\infty} \om_0\bigl (\Xi(s_0, t)\xi_0,\,\partial_s\Xi(s_0, t)\xi_0\bigr ).
  \end{aligned}
\end{equation}

On the other hand, for $\xi_0$ as above and for all $s$ from a sufficiently small neighbourhood of $s_0$ there exists $\xi^-(s)\in\Lambda^-(s_0)$ such that $\xi_0+\xi^-(s)\in\Lambda^+(s)$, i.e.,
\begin{equation*}
 \lim\limits_{t\to +\infty}\Xi(s_0, t)\xi^{-}(s)=0\quad \text{and}\quad  \lim\limits_{t\to -\infty}\Xi(s, t)(\xi_0+\xi^{-}(s))=0.
\end{equation*}
For $t\le -T$ we must have $\Xi(s, t)(\xi_0+\xi^{-}(s))=\sum_{j=1}^n c_j(s)e^{\l_jt}$, where $\l_1,\dots,\l_n$ are positive eigenvalues of the matrix $J_0H_+$. Hence, $ \partial_s\Xi(s_0, t)\xi_0 + \Xi(s_0, t) \partial_s\xi^-(s_0)\to 0$ as $t\to -\infty$ and this in turn implies
\begin{equation}\label{Eq_AuxForCrossigFormOfLagrPair1}
  \begin{aligned}
     \om_0\bigl (\xi_0, \partial_s\xi^-(s_0)\bigr ) &=\om_0\bigl (  \Xi(s_0, t)\xi_0,  \Xi(s_0, t)\partial_s\xi^-(s_0) \bigr )\\%
         &=-\lim\limits_{t\to -\infty} \om_0 \bigl (  \Xi(s_0, t)\xi_0,  \partial_s\Xi(s_0, t)\xi_0 \bigr ).
     \end{aligned}
\end{equation} 
Similarly, there also exists $\xi^+(s)\in\Lambda^+(s_0)$ for all $s$ sufficiently close to $s_0$ such that $\xi_0+\xi^+(s)\in\Lambda^-(s)$. Arguing as above, wee see that %
\begin{equation}\label{Eq_AuxForCrossigFormOfLagrPair2}
 \om_0\bigl (\xi_0, \partial_s\xi^+(s_0)\bigr )=%
     -\lim\limits_{t\to +\infty} \om_0 \bigl (  \Xi(s_0, t)\xi_0,  \partial_s\Xi(s_0, t)\xi_0 \bigr ).
\end{equation} 
Since by definition %
$\Gamma(\Lambda^+,\Lambda^-, s_0)\xi_0 =\om_0\bigl (\xi_0, \partial_s\xi^-(s_0)\bigr ) - \om_0\bigl (\xi_0, \partial_s\xi^+(s_0)\bigr )$, 
combining (\ref{Eq_CrossingFormForFamily})-(\ref{Eq_AuxForCrossigFormOfLagrPair2})  we finally obtain $\Gamma(\Lambda^+,\Lambda^-, s_0)\xi_0=\Gamma(A, s_0)\xi_0$. This finishes the proof.  
\end{proof}

For any $p>2$ consider the space
\begin{align*}
\mathcal B=\bigl\{
 u\in W^{1,p}_{loc}(\R^2; M)\mid\  & \exists\; R>0, \xi_\pm\in W^{1,p}(\gamma_\pm^*TM),\text{ and } \eta_\pm\in W^{1,p}(\R^2; T_{m_\pm}M)\bigr .\\
& \text{ s.t. }\bigl .u=\exp_{\gamma_\pm}\xi_\pm \text{ for } \mp s>R\text{ and } u=\exp_{m_\pm}\eta_\pm \text{ for } \mp t>R \bigr\}.
\end{align*}
One can construct an atlas on $\mathcal B$ similarly to~\cite[Theorem~3]{Floer:88_UnregGradientFlow}.  Thus $\mathcal B$ is a Banach manifold. Observe that for $u\in\mathcal B$ we have $T_u\mathcal B=W^{1,p}(\R^2; u^*TM)$.

Let $\mathcal F\rightarrow \mathcal B$ be the vector bundle with the fiber $\mathcal F_u=L^p(\R^2; u^*TM)$. Then the map $\Sigma$ can be interpreted as a section of $\mathcal F$. Clearly, any solution of $\Sigma(u)=0$ is a smooth map. By Corollary~\ref{Cor_FiniteEnergyImplyBCs} the zero locus of $\Sigma$ coincides with $\mathcal M(\gamma_-;\gamma_+)$. Notice also that the covariant derivative of $\Sigma$ at the point $u$ can be identified with the map
\[\label{Eq_AuxDuSigmaBanachSpaces}
D_u\Sigma\colon W^{1,p}(\R^2; u^*TM)\rightarrow L^p(\R^2; u^*TM),
\] 
which is Fredholm. We summarise these observations in the following proposition.
\begin{prop}
  The zero locus of $\Sigma\in\Gamma(\mathcal B;\mathcal F)$ is the space of solutions of~\eqref{Eq_PseudoholomPlaneWithHamPert_mod}--\eqref{Eq_PseudoholomPlaneWithHamPert_BC_s}. Moreover, for each zero $u$ the covariant derivative $D_u\Sigma$ is Fredholm.
\end{prop}


\section{A gauge theory on 5-manifolds}\label{Sect_5dGaugeTheory}

Let $E$ be a five-dimensional oriented Euclidean vector space with a preferred vector $\mathrm v\in E$ of unit norm. Let $\eta (\cdot)=\langle\mathrm v, \cdot\rangle$ denote the corresponding 1-form. Then the linear map%
\begin{equation*}
 T_\eta\colon\Lambda^2 E^*\longrightarrow \Lambda^2 E^*,\quad \om\mapsto *(\om\wedge\eta)
\end{equation*}
has three eigenvalues $\{ -1,0, +1 \}$ and the space $\Lambda^2 E^*$ decomposes as the direct sum of the corresponding eigenspaces:%
\begin{equation*}
 \Lambda^2 E^*\cong \Lambda_-^2 E^*\oplus \Lambda_0^2 E^*\oplus \Lambda_+^2 E^*.
\end{equation*}
Indeed, denote by $H$ the orthogonal complement of $\mathrm v$. Then $\Lambda^2E^*\cong\Lambda^2 H^*\oplus H^*$ and one easily checks that the following subspaces $\Lambda^2_\pm H^*$ and $H^*$ are eigenspaces of $T_\eta$, where $\Lambda^2_\pm H^*$ denote the eigenspaces of the four-dimensional Hodge star operator. In other words, 
$\Lambda^2_\pm  E^*\cong\Lambda^2_\pm H^*$ and $\Lambda^2_0 E^*\cong H^*$. 

Identify the Clifford algebra of $E$ with $\Lambda E$ and recall the following description of the Clifford multiplication%
   \begin{align*}
       &Cl\colon E^*\otimes\Lambda E^*\longrightarrow\Lambda E^*, \quad  & Cl=Cl'+Cl'',\\
       &Cl'\colon E^*\otimes\Lambda^p E^*\cong E\otimes\Lambda^p E^*\xrightarrow{\ c\ }\Lambda^{p-1} E^*, & c(e\otimes\om)=-\imath_e\om,\\
       &Cl''\colon E^*\otimes\Lambda^p E^*\xrightarrow{\ \cdot\,\wedge\,\cdot}\Lambda^{p+1} E^*. &
   \end{align*}
In particular, by restriction we get a map $ Cl'\colon E^*\otimes\Lambda^2_+E^*\longrightarrow E^*$, which is essentially the four-dimensional homomorphism $H^*\otimes\Lambda^2_+ H^*\longrightarrow H^*$. 

Observe that $\Lambda^2_+ H^*$ has a natural structure of a Lie algebra as a three-dimensional oriented Euclidean vector space. For an arbitrary Lie algebra $\mathfrak g$ denote $V=\Lambda^2_+ H^*\otimes\mathfrak g$ and consider the linear map $\sigma\colon V\otimes V\rightarrow V,\ \sigma=\tfrac 12 [\cdot\, ,\cdot]_{\Lambda^2_+H^*}\otimes [\cdot\, ,\cdot]_{\mathfrak g}$. 
Choosing a Lie algebra isomorphism $\Lambda^2_+H^*\cong\mathbb R^3$, for  $\xi=e_1\otimes\xi_1+e_2\otimes\xi_2+e_3\otimes\xi_3$ we obtain %
\begin{equation*}
 \sigma(\xi,\xi)= e_1\otimes [\xi_2,\xi_3]+e_2\otimes [\xi_3, \xi_1]+ e_3\otimes [\xi_1, \xi_2].
\end{equation*}
\smallskip

Let $(W^5, g)$ be an arbitrary oriented Riemannian five-manifold with a preferred vector field $v$ of pointwise unit norm. Denote $\eta(\cdot)=g(v,\cdot)\in\Om^1(W)$ and $\mathcal H=\mathrm{ker}\, \eta\subset TW$. As described above, we have the following splittings:%
\begin{align*}
   \Om^1(W) &=\Om^1_h(W)\oplus \Om^0(W)\eta,\qquad  \Om^1_h(W)=\Gamma(\mathcal H^*),\\
   \Om^2(W) &=\Om^2_-(W)\oplus\Om^2_0(W)\oplus\Om^2_+(W).
\end{align*}
Let $P\rightarrow W$ be a principal $G$-bundle, where $G$ is a compact Lie group. Denote by $\mathcal A(P)$ the space of connections on $P$ and by $ad\, P$ the adjoint bundle of Lie algebras. Consider the following equations for a pair $(A,B)\in\mathcal A(P)\times\Om^2_+(ad\, P)=\mathcal B$:
\begin{equation}\label{Eq_Main}  
 \begin{aligned}
   &\imath_vF_A-\delta_A^+\, B=0,\\
   & F_A^+-\nabla^A_v\,B-\sigma(B,B)=0,
 \end{aligned}
\end{equation}
where the operator $\delta_A^+\colon\Om^2_+ (ad\, P)\rightarrow\Om^1_h(ad\, P)$ is defined by the composition%
\begin{equation*}
   \delta_A^+\colon\Gamma(\Lambda^2_+\mathcal H^*\otimes ad\, P)\xrightarrow{\ \nabla^{LC,\, A}\ }%
           \Gamma(T^*W\otimes \Lambda^2_+\mathcal H^*\otimes ad\, P)\xrightarrow{\ Cl'\otimes\, id\ }\Gamma(\mathcal H^*\otimes ad\, P).
\end{equation*}
Here $\nabla^{LC,\, A}$ denotes the tensor product of $A$ and the connection on $\Lambda^2_+\mathcal H^*$ induced by the Levi-Civita connection (we do not assume that $\Lambda^2_+\mathcal H^*$ is preserved by the Levi-Civita connection). It is convenient to define a map $\Phi\colon \mathcal B\rightarrow \Om^1_h(ad\, P)\times \Om^2_+(ad\, P)$ by the left hand side of equations (\ref{Eq_Main}).

\begin{rem}
 Equations~\eqref{Eq_Main} were independently discovered by Witten~\cite{Witten:12_FivebranesAndKnots_QT} from a different perspective. A partial case with $B\equiv 0$ has been studied by Fan~\cite{Fan:96}.
\end{rem}

\begin{rem}\label{Rem_5dEqnsAsSpin7Instantons}
 The total space of $\Lambda^2_+\mathcal H\rightarrow W$ is an eight-manifold equipped with a natural $Spin(7)$-structure, which is induced by the Riemannian metric and orientation on $W$. This $Spin(7)$-structure can be constructed using the technique of~\cite{BryantSalamon:89}. Then, following the line of argument in~\cite{Haydys:09_GaugeTheory_jlms}, one can show  that solutions of equations (\ref{Eq_Main})  correspond to $Spin(7)$-instantons on $\Lambda^2_+\mathcal H$ invariant along each fibre.
\end{rem}

The gauge group $\mathcal G(P)$ acts on the configuration space $\mathcal B$ on the right%
\begin{equation*}
   (A,B)\cdot g=(A\cdot g, \ ad_{g^{-1}}\, B),\qquad g\in\mathcal G(P),
\end{equation*}
where $g$ acts on the first component by the usual gauge transformation. The infinitesimal action at a point $(A,B)$ is given by the map%
\begin{equation*}
   K\colon\Om^0(ad\, P)\longrightarrow \Om^1(ad\, P)\oplus\Om^2_+(ad\, P),\qquad \xi\mapsto \bigl (d_A\xi,\ [B, \xi]\bigr ).
\end{equation*}
Notice also that the map $\Phi$ is $\mathcal G(P)$-equivariant.

The standard computation yields%
\begin{equation*}
   \delta\Phi_{(A,\, B)}
      \begin{pmatrix}
       \alpha\\ \beta
      \end{pmatrix}=
   \begin{pmatrix}
       \imath_v(d_A\alpha) -\delta_A^+\,\beta +\alpha\cdot B\\
       d_A^+\alpha-\nabla^A_v\,\beta - [\alpha(v), B] -2\sigma(B,\beta)
   \end{pmatrix},
\qquad (\alpha, \beta)\in T_{(A, B)}\mathcal B,
\end{equation*}
where the term $\alpha\cdot B\in\Om^1_h(ad\, P)$ is constructed algebraically from $\alpha$ and $B$, namely $\alpha\cdot B=Cl'\otimes [\cdot\, ,\cdot]_{\mathfrak g}(\alpha\otimes B)$. Thus we get the deformation complex at the point $(A,B)$:
\begin{equation}\label{Eq_5dDeformationComplex}
   0\rightarrow\Om^0(ad\, P)\xrightarrow{\ K\ }\Om^1(ad\, P)\oplus\Om^2_+(ad\, P)\xrightarrow{\ \delta\Phi\ }%
        \Om^1_h(ad P)\oplus \Om^2_+(ad\, P)\rightarrow 0.
\end{equation}

\begin{lem}
 Deformation complex~(\ref{Eq_5dDeformationComplex}) is elliptic.
\end{lem}

The statement of this Lemma follows immediately from Remark~\ref{Rem_5dEqnsAsSpin7Instantons}. Alternatively, one can consider equations~(\ref{Eq_Main}) on $\mathbb R^5$ and show that the symbol of  $K^*+\delta\Phi$ is modelled on the octonionic multiplication. We omit the details.


\section{Dimensional reductions}\label{Sect_DimReductions}

Before turning our attention to the dimensional reductions of equations~(\ref{Eq_Main})  a little digression is in place. Suppose a Lie group $\mathcal G$ acts \emph{freely} and isometrically on a Riemannian manifold $M$. Identify a $\mathcal G$-invariant function $ f\colon M\rightarrow\mathbb R$ with a function $\hat f\colon M/\mathcal G\rightarrow\mathbb R$. Then critical points of $\hat f$ correspond to orbits of solutions of the equation $\grad f= K_\xi$, where $\xi\in Lie(\mathcal G)$ and $K_\xi$ is the Killing vector field corresponding to $\xi$. But the invariance of $f$ implies %
$\langle \grad f, K_{\xi'}  \rangle=0$ for any $\xi'\in Lie(\mathcal G)$ so that we necessarily have $\grad f=0$ for any point on $M$ projecting to a critical point of $\hat f$. 

Similarly, a curve $m\colon\mathbb R\rightarrow M$ projects to an antigradient flow of $\hat f$ if and only if there exists $\xi\colon\mathbb R\rightarrow Lie(\mathcal G)$ such that%
\begin{equation}\label{Eq_EquivFlowEquation}
 \dot m= -\grad f+K_\xi.
\end{equation}
The Lie group $\{ g\colon\mathbb R\rightarrow \mathcal G \}$ acts on solutions of equation~(\ref{Eq_EquivFlowEquation}) and the orbits are in bijective correspondence with the antigradient flow lines of $\hat f$. Furthermore, we may consider only those solutions, which are horizontal with respect to the natural connection. This gives a bijection between ordinary flow lines of $f$ modulo $\mathcal G$ and flow lines of $\hat f$. 

The upshot is that $\mathcal G$-invariance of $f$ implies that  equation~(\ref{Eq_EquivFlowEquation}) is equivalent to the ordinary antigradient flow equation of $f$.  It will be important to switch freely between these two approaches in an infinite-dimensional setup. The reasons will be clear below.

\subsection{Dimension four}

Let $X$ be a closed oriented Riemannian four-manifold. Below we consider equations~(\ref{Eq_Main})  on $(W,v)=(X\times\mathbb R_t, \frac \partial{\partial t})$ endowed with the product metric.

Denote  by $pr\colon X\times\mathbb R\rightarrow X$ the canonical projection and set $P=pr^*P_X$, where $P_X\rightarrow X$ is a principal $G$-bundle. Think of $B\in\Om^2_+(X\times\mathbb R; pr^*ad\, P_X)$ as a map %
$b\colon\mathbb R\rightarrow \Om^2_+(X; ad\, P_X)$. Similarly $A\in\mathcal A(pr^*P_X)$ can be seen as a map %
$(a, c)\colon\mathbb R\rightarrow\mathcal A(P_X)\times\Om^0(ad\, P_{X})$, where $c$ is the Higgs field. Then equations~(\ref{Eq_Main}) are easily seen to become%
\begin{equation}\label{Eq_GradFlowOnXR}
  \begin{aligned}
   &\dot a= \delta_a^+  b + d_a c, \\
   &\dot b=F_a^+-\sigma(b, b) - [c,b], 
 \end{aligned}
\end{equation}
where $\delta_a^+=(d_a^+)^*$.  These equations turn out to be the antigradient flow equations of some function. Indeed, consider the function%
\begin{equation*}
 h\colon \Lambda^2_+ H^*\otimes\mathfrak g\rightarrow\mathbb R,\qquad h(\mathrm w)=\frac 13\langle \mathrm w, \sigma(\mathrm w)\rangle.
\end{equation*}
Choose an isomorphism $\Lambda^2_+H^*\cong\mathbb R^3$ and write $\mathrm w=\sum_{i=1}^3 e_i\otimes\xi_i$. Then we have %
$h(\mathrm w)=\langle\xi_1, [\xi_2,\xi_3]\rangle$ and therefore %
$\mathrm{grad}\, h (\mathrm w)=\sigma (\mathrm w) $. Since $h$ is equivariant with respect to both $SO(3)$ and $G$, we obtain a well-defined map  $\Om^2_+(ad\, P_X)\rightarrow C^\infty(X)$ denoted by the same letter.

Denote  $\mathcal B=\mathcal A(P)\times \Om^2_+(ad\, P)/\mathcal G(P)$. As usual, $\mathcal B^*\subset \mathcal B$  denotes the quotient space of irreducible points. The negative $L^2$--gradient of the function %
\begin{equation*}
 U\colon\mathcal B\rightarrow\mathbb R,\qquad%
    U(a,b)=-\langle F_a^+, b\rangle_{L^2}+\int_X h(b)\, vol_X
\end{equation*}
is $(\delta_a^+ b,\, F_a^+-\sigma(b,b))$. Hence,  assuming there are no reducible solutions, equations~(\ref{Eq_GradFlowOnXR})  represent the antigradient flow equations of the function $\hat U\colon \mathcal B^*\rightarrow \mathbb R$.

We summarize our computations in the following proposition.

\begin{prop}
 If there are no reducible solutions, equations~(\ref{Eq_GradFlowOnXR})  represent antigradient flow equations of the function $\hat U\colon\mathcal B^*\rightarrow \mathbb R$.\qed
\end{prop}

The critical points of the function $U$ are solutions of the Vafa--Witten equations~\cite{VafaWitten:94_StrongCouplingTest}:%
\begin{align}
   & \delta_a^+ b+d_a c=0, \notag\\
   & F_a^+-\sigma(b, b) +[b,c]=0.\notag
 \end{align} 
These equations are elliptic and the expected dimension of the moduli space is zero.

As we have seen, the $\mathcal G(P)$-invariance of $U$ implies that for each irreducible solution of the Vafa--Witten equations we have $(d_ac, [b,c])=0$, i.e. in the absence of reducible solutions the above equations are equivalent to%
\begin{equation}\label{Eq_RedVW}
 \begin{aligned}
   & \delta_a^+b =0, \\
   & F_a^+-\sigma(b, b) =0.
  \end{aligned}
\end{equation}
Notice that the Weitzenb\"ock formula %
\begin{equation*}
2d_a^+ \delta_a^+=(\nabla^a)^*\nabla^a -2W^+\, +\frac s3\,  + \sigma(F_a^+,\cdot ),
\end{equation*}
yields
\begin{equation*}
 \begin{aligned}
   4\| \delta_a^+b\|^2+ \| F_a^+-\sigma(b,b)\|^2 &= 2\| \nabla^a b \|^2- 4\langle W^+(b), b\rangle +\frac 23\langle s b, b\rangle +2\langle F_a^+, \sigma (b,b)\rangle\\
				&\phantom{=} +\| F_a^+\|^2 + \| \sigma(b,b)\|^2-2\langle F_a^+, \sigma (b,b)\rangle\\
                &=2\| \nabla^a b \|^2- 4\langle W^+(b), b\rangle +\frac 23\langle s b, b\rangle +\| F_a^+\|^2 + \| \sigma(b,b)\|^2.
 \end{aligned}
\end{equation*}

\begin{prop}[\cite{VafaWitten:94_StrongCouplingTest}]\label{Prop_VWVanishingThm}
 If the operator $-W^++\frac 16 s$ is pointwise non-negative definite on $\Lambda^2_+T^*X$, then for any irreducible solution $(a,b)$ of the Vafa--Witten equations the following holds: $F_a^+=0,\ \nabla^a b=0$.\qed
\end{prop}


\subsection{Dimension three}\label{Subsec_Dim3}

In this section various forms of equations~(\ref{Eq_Main}) are studied on $Y^3\times\mathbb R^2$, where $Y$ is a closed oriented Riemannian three-manifold. 

Just like in the instanton Floer theory, consider solutions of~(\ref{Eq_RedVW}) on $X=Y\times\mathbb R$. Assuming  $a$ is in a temporal gauge, we obtain the following system of equations%
\begin{equation}\label{Eq_GradFlowReCS}
\begin{aligned}
   &\dot a=-\! \ast\!(F_a-  \tfrac 12[b\wedge b]),\\
   &\dot b=\ast d_a b,\\
    &0= \delta_{a}  b,
 \end{aligned}
\end{equation}
where $(a,b)$ is interpreted as a curve in $\mathcal A(P)\times\Om^1(ad\, P)\cong T^*\mathcal A(P)$. Here we have also used the isomorphism $\Gamma(\pi^*T^*Y)\cong\Om^2_+  (Y\times\mathbb R), \ \om\mapsto\tfrac 12 (\ast_{\ssst 3}\om +ds\wedge\om)$, where $\pi\colon Y\times\mathbb R\rightarrow Y$ is the projection.

Observe that $T^*\mathcal A(P)$ is a (flat) K\"ahler manifold and the action of the gauge group is Hamiltonian.  The momentum map is given by%
\begin{equation}\label{Eq_MomentMapOfGaugeGp}
 \mu\colon T^*\mathcal A(P)\rightarrow \Om^0(ad\, P),\qquad \mu (a,b)=\delta_ab.
\end{equation}
Denote $N=\mu^{-1}(0)=\{(a,b)\; |\ \delta_a b=0\}\subset T^*\mathcal A(P)$. It follows from the very definition of the momentum map that $d\mu$ is surjective  at $(a,b)$ if and only if the gauge group acts locally freely at $(a,b)$. Therefore, the subset $N^*$ consisting of all irreducible points of $N$ is a submanifold. Hence, $N^*/\mathcal G(P)$ is a K\"ahler manifold.

Consider the map%
\begin{equation*}
 f_0\colon\mathcal A(P)\times\Om^1(ad\, P)\rightarrow \mathbb R/\mathbb Z,\qquad %
 f_0(a,b)= 8\pi^2\vartheta (a) - \tfrac 12 \langle b,\, *d_a b \rangle_{L^2},
\end{equation*}
where $\vartheta$ is the Chern--Simons function. It is easy to check that the vector field %
$\grad f_0=\bigl (*(F_a-  \frac 12[b\wedge b]\bigr ),\; -*d_a b)$ is tangent to $N^*$ at each point. Therefore critical points of the restriction of $f_0$ to $N^*$ are solutions of Hitchin's equations\footnote{Hitchin studied these equations in the case of two-dimensional base manifolds.}~\cite{Hitchin:87}:%
\begin{equation}\label{Eq_Hitchins}
 \begin{aligned}
& F_a- \tfrac 12[b\wedge b]=0,\\
& d_a b=0,\\
& \delta_{a}  b =0.
 \end{aligned}
\end{equation}
More accurately, in the same manner as described at the beginning of this section, orbits of irreducible solutions to~(\ref{Eq_Hitchins}) correspond to critical points of $\hat f_0\colon N^*/\mathcal G(P)\rightarrow \mathbb R$. Similarly, orbits of~\eqref{Eq_GradFlowReCS} correspond to the flow lines of $\hat f_0$.

\begin{rem}\label{Rem_GandGcConnections}
Denote by $G^c$ the complexified Lie group and by $\mathcal P=P\times_GG^c$ the principal $G^c$-bundle associated with $P$. Any connection on $\mathcal P$ can be written in the form $\mathscr A=a+ib$, where $(a,b)\in\mathcal A(P)\times \Om^1(ad\, P)$. Conversely, any pair $(a,b)$ combines to a $G^c$ connection $\mathscr A$. Then $\mathscr A$ is flat if and only if the first two equations of~\eqref{Eq_Hitchins} are satisfied. The last equation, i.e. the vanishing of the moment map, has been analyzed in~\cite{Donaldson:87_TwistedHarmonicMaps, Corlette:88_FlatGbundles}. 
\end{rem}

\begin{rem}
 Hitchin's equations can be obtained from $SU(3)$ anti-self-duality equations along similar lines to those outlined in Remark~\ref{Rem_5dEqnsAsSpin7Instantons}. Namely, the total space of $T^*Y$ is equipped with an $SU(3)$-structure. Then $SU(3)$-instantons invariant along each fiber are solutions of Hitchin's equations.
\end{rem}

We can also consider equations~\eqref{Eq_Main} on $W=\mathbb R_t\times Y\times\mathbb R_s$ with $v=-\tfrac\partial{\partial s}$. Write $A=a+e\,ds+ c\,dt$, where $a$ is a family of connections on $P\rightarrow Y$. Consider first only $t$-invariant solutions with $c=e=0$. A computation yields the following system:%
\begin{equation}\label{Eq_GradFlowImCS}
 \begin{aligned}
&\dot a=-\ast\! d_a b,\\
&\dot b=-\ast\! (F_a - \tfrac 12[b\wedge b]),\\
 &0 =\delta_{a}  b,
 \end{aligned}
\end{equation}
where the dots denote the derivative with respect to the variable $s$. Equations~\eqref{Eq_GradFlowImCS} and~\eqref{Eq_GradFlowReCS}  appeared in~\cite{KapustinWitten:07_ElectrMagnDuality} for the first time  and were further studied in~\cite{Witten:10_AnalyticContinuation,Witten:11_NewLookAtPathIntegral}.

Consider the function
\begin{equation*}
 f_1\colon T^*\mathcal A(P)\rightarrow\mathbb R,\qquad%
 f_1(a,b)=\langle F_a,\, *b\rangle_{L^2} - \int_Y h(b)\, vol_Y  .
\end{equation*}
Since  $\grad f_1=(* d_a b,\, *(F_a-\frac 12[b\wedge b]))$ is tangent to $N^*$ at each point we conclude that the moduli of solutions to equations~\eqref{Eq_GradFlowImCS} correspond to antigradient flow lines of $\hat f_1\colon N^*/\mathcal G(P)\rightarrow \mathbb R$.

Let us examine the functions $f_0$ and $f_1$ more closely. Since $\grad f_1=J\grad f_0$, where $J$ is the constant complex structure on $T^*\mathcal A(P)\cong \Om^1(ad\, P)\otimes\mathbb C$, we obtain that the function $f=f_0+if_1$ is $J$-holomorphic. Writing  $(a,b)$ as a $G^c$-connection $\mathscr A$ as in Remark~\ref{Rem_GandGcConnections} it is easy to check that $f$ is the complex Chern--Simons functional
\begin{equation*}
\CS(\mathscr A)=%
\frac 12\int_Y\left ( \langle \mathscr A\wedge d\mathscr A \rangle +  
    \tfrac 13\langle \mathscr A\wedge [\mathscr A\wedge \mathscr A] \rangle\right ).
\end{equation*}
Here we interpret $\mathscr A$ as a $\mathfrak{g}_{\mathbb C}$-valued 1-form on $Y$, and $\langle\cdot,\cdot\rangle\colon \mathfrak{g}_{\mathbb C}\otimes\mathfrak{g}_{\mathbb C}\rightarrow\mathbb C$ denotes the $\mathbb C$-linear extension of the scalar product on $\mathfrak g$.

\medskip

Further, let us consider equations~(\ref{Eq_Main}) on $(W, v)=(Y\times \mathbb R^2_{s,t}\, , \tfrac\partial{\partial t}\,)$. The standard reduction procedure yields the following system%
\begin{equation}\label{Eq_5dEquationsAsSVE}
 \begin{aligned}
     &\partial_s a-\partial_tb +[b,c] -d_a e +\ast\duzhky{F_a-\tfrac 12[b\wedge b]}=0,\\
    &\partial_t a+\partial_s b-[b,e]-d_a c- *d_a b=0,\\
    &\partial_t e-\partial_s c+[c, e]+\delta_a b=0.
 \end{aligned}
\end{equation}
Here $a$ is a connection on the pull-back of $P=P_Y$ to $Y\times\mathbb R^2$, $b$ is a 1-form and $c,\ e$ are 0-forms with values in the adjoint bundle of Lie algebras. It is easy to check that these equations are symplectic vortex equations~\cite{SalamonEtAl:00} with a Hamiltonian perturbation for the following data: The target space is $T^*\mathcal A(P)$ equipped with the Hamiltonian action of the gauge group $\mathcal G(P)$, $u=(a,\, b)\colon\mathbb R^2\rightarrow T^*\mathcal A(P),\ A=e\,ds+ c\, dt$, and the perturbation is %
$\sigma=\tfrac 12\im (f dz)=\tfrac 12 (f_0\, dt +f_1\, ds)$.

Notice also that we are free to rotate the coordinates $s$ and $t$ or, equivalently, to rotate the initial vector field $v={\partial_t}$. This is in turn equivalent to the choice of the Hamiltonian perturbation $\sigma=\tfrac 12\im (e^{i\theta}f dz)$ and the resulting equations are  
\begin{equation}\label{Eq_5dEquationsAsSVEWithParameter}
 \begin{aligned}
    &\partial_s a-\partial_tb +[b,c] -d_a e -\sin\theta\, *\! d_ab +\cos\theta\,  \ast\!\duzhky{F_a-\tfrac 12[b\wedge b]} =0,\\
    &\partial_t a+\partial_s b-[b,e]-d_a c-\cos\theta\, *\!  d_a b - \sin\theta\, \ast\!\duzhky{F_a-\tfrac 12[b\wedge b]} =0,\\
    &\partial_t e -\partial_s c+[c, e]+\delta_a b=0.
 \end{aligned}
\end{equation}

\begin{rem}\label{Rem_AdiabaticLimitFor5dEqns}
The above description of equations~\eqref{Eq_5dEquationsAsSVEWithParameter} is analogous to the interpretation of the anti-self-duality equations on $\mathbb C\times\Sigma$ as symplectic vortex equations~\cite{SalamonEtAl:00}. A new phenomenon here is the appearance of the Hamiltonian perturbation. Notice also that the adiabatic limit procedure as in~\cite{SalamonEtAl:00, GaioSalamon:05} for equations~(\ref{Eq_5dEquationsAsSVEWithParameter}) yields (at least formally) holomorphic planes to $T^*\mathcal A(P)\sympred\mathcal G(P)$ with a Hamiltonian perturbation. 
\end{rem}

For solutions of equations~\eqref{Eq_5dEquationsAsSVEWithParameter} invariant with respect to $s$ we obtain the following system
\begin{equation}\label{Eq_GradFlowCSwithTheta}
 \begin{aligned}
    \dot a &=\phantom{-}\cos\theta\, *\!  d_a b +\sin\theta\, \ast\!\duzhky{F_a-\tfrac 12[b\wedge b]} + d_a c +[b,e],\\
    \dot b &=-\sin\theta\, *\! d_ab +\cos\theta\,  \ast\!\duzhky{F_a-\tfrac 12[b\wedge b]} -d_a e+[b,c], \\
    \dot e &=-\delta_a b +  [e, c].
 \end{aligned}
\end{equation}
Notice that if $c=e=0$ we obtain equations~\eqref{Eq_GradFlowReCS} and~\eqref{Eq_GradFlowImCS} for $\theta=-\pi/2$ and $\theta=\pi$, respectively.

It will be helpful in the sequel to consider equations~\eqref{Eq_GradFlowCSwithTheta} from a more abstract point of view. Namely, let $(M, \om)$ be a symplectic manifold. Assume a Lie group $\mathcal G$ acts on $M$ in a Hamiltonian manner. Denote by  $\mu\colon M\rightarrow \mathfrak G=Lie(\mathcal G)$ the  corresponding moment map. Let $f\colon M\to\C$ be a $\mathcal G$-invariant $J$-holomorphic function, where $J$ is a $\mathcal G$-invariant almost complex structure on $M$.  Consider the following equations for a curve $(\gamma,\xi, \eta)$ in $M\times\mathfrak G\times\mathfrak G$:
\begin{equation}\label{Eq_AbstrGradFlowSymplRed}
 \begin{aligned}
  \dot\gamma &=\sin\theta\,\grad f_0 (\gamma) +\cos\theta\,\grad f_1(\gamma) + K_\xi(\gamma) -JK_\eta (\gamma),\\
   \dot\eta &= -\mu(\gamma) - [\xi,\eta],
 \end{aligned}
\end{equation}
where $K$ is the Killing vector field. Clearly, we obtain equations~\eqref{Eq_GradFlowCSwithTheta} from~\eqref{Eq_AbstrGradFlowSymplRed} putting  $M=T^*\mathcal A(P)$. 

Further, observe that for any  $\zeta,\rho\in\mathfrak G$ the following equalities hold: %
\begin{equation*}
 d \bigl (\om(K_\zeta, K_\rho)\bigr )= d\,\imath_{K_\rho}(\imath_{K_\zeta}\om)=%
    \mathcal L_{K_\rho}(\imath_{K_\zeta}\om)- \imath_{K_\rho} d(\imath_{K_\zeta}\om)=%
    \imath_{K_{[\rho,\zeta]}}\om = - d\,\langle \mu, [\zeta,\rho]\rangle.
\end{equation*}
Here the second equality follows from Cartan's equation. Hence, $\langle \mu, [\zeta,\rho]\rangle=-\om(K_\zeta, K_\rho)=g(K_\zeta, JK_\rho)$.
Therefore for any solution of~\eqref{Eq_AbstrGradFlowSymplRed} we have%
\begin{equation}\label{Eq_AuxNegativeDerivative}
 \begin{aligned}
   \frac d{dt}\langle \mu(\gamma),\eta\rangle &=\om(K_\eta,\dot\gamma) +\langle \mu,\dot\eta\rangle=%
    g(JK_\eta, K_\xi)-g(JK_\eta, JK_\eta)- \langle \mu,\mu\rangle - \langle \mu, [\xi, \eta]\rangle\\%
   &= -g(K_\eta, K_\eta)- \langle \mu,\mu\rangle\le 0.
 \end{aligned}
\end{equation}
Here the first equality follows from the definition of the momentum map, the second one from equations~\eqref{Eq_AbstrGradFlowSymplRed} and the $\mathcal G$-invariance of $f$, and the last one from the equation $\langle \mu, [\xi,\eta]\rangle=g(K_\xi, JK_\eta)$. Hence, for any solution of equations~\eqref{Eq_AbstrGradFlowSymplRed} the function $\langle \mu(\gamma),\eta\rangle$ is non-increasing. 

We will be interested below in solutions $(\gamma, \xi,\eta)$ of~\eqref{Eq_AbstrGradFlowSymplRed} satisfying the condition%
\begin{equation}\label{Eq_CondAtInftyForAbstrFlow}
 (\gamma, \xi,\eta)\longrightarrow (m_\pm, 0, 0)\qquad\text{as}\quad t\to \mp\infty,
\end{equation}
where $m_\pm$ are critical points of $f$. For any such solution $\langle \mu(\gamma),\eta\rangle$ vanishes at $\pm\infty$ and hence vanishes everywhere. Then from~\eqref{Eq_AuxNegativeDerivative} we conclude that $\eta$ and $\mu\comp\gamma$ vanish everywhere, i.e. under condition~\eqref{Eq_CondAtInftyForAbstrFlow} equations~\eqref{Eq_AbstrGradFlowSymplRed} reduce to
\begin{equation*}
  \dot\gamma =\sin\theta\,\grad f_0 +\cos\theta\,\grad f_1 + K_\xi,\qquad \mu(\gamma)=0.
\end{equation*}
 
From the discussion at the beginning of Section~\ref{Sect_DimReductions} we obtain that these equations are equivalent to
\begin{equation}\label{Eq_ReducedAbstrGradFlowSymplRed}
  \dot\gamma =\sin\theta\,\grad f_0 +\cos\theta\,\grad f_1,\qquad \mu(\gamma)=0.
\end{equation}
Summing up, we have that under condition~\eqref{Eq_CondAtInftyForAbstrFlow} systems~\eqref{Eq_AbstrGradFlowSymplRed} and~\eqref{Eq_ReducedAbstrGradFlowSymplRed} are equivalent. Applying this conclusion in the case $M=T^*\mathcal A(P)$ we obtain that for $\theta=-\pi/2$ equations~\eqref{Eq_GradFlowCSwithTheta} together with the condition %
\begin{equation*}
 (a,b, c, e)\longrightarrow (a_\pm, b_\pm,0,0)\qquad\text{as}\quad t\to\pm\infty,
\end{equation*}
where $(a_\pm, b_\pm)$ are solutions of Hitchin's equations, are equivalent to equations~\eqref{Eq_GradFlowReCS} together with $(a,b)\to (a_\pm, b_\pm)$ as $t\to\mp\infty$. The upshot is that while equations~\eqref{Eq_GradFlowReCS}  and~\eqref{Eq_GradFlowCSwithTheta} with $\theta=-\pi/2$ are essentially equivalent, only the latter are elliptic. 

\begin{rem}
 One obtains an elliptic form of Hitchin's equations on a three manifold by considering solutions of equations~\eqref{Eq_Main} on $(Y\times\mathbb R^2, \partial_t)$ invariant along $\mathbb R^2$. The corresponding equations are easily obtained from~\eqref{Eq_5dEquationsAsSVE}. 
\end{rem}



\section{Invariants}\label{Sect_Invariants}

In this section we outline constructions of invariants assigned to five-, four-, and three- manifolds arising from gauge theories described in the preceding sections.  It is clear that the constructions described below need an appropriate analytic justification. We postpone this to subsequent papers and restrict ourselves to some examples.  Throughout this section  the coefficient ring is $\mathbb Z/2\mathbb Z$ in all constructions for the sake of simplicity. 

The expected dimension of the moduli space of solutions of equations~(\ref{Eq_Main}) for closed five-manifolds is zero. Therefore, assuming compactness and transversality, an algebraic count associates a number to closed five-manifolds. More accurately,  this number depends on the isomorphism class of $P$ and on the class of the vector field $v$ in $\pi_0(\mathfrak X_0(W))$, where $\mathfrak X_0(W)$ denotes the space of all vector fields on $W$ without zeros. 

\medskip

Let us now consider the dimension four. The corresponding construction is very similar to the instanton Floer theory, so we are very brief here. 
Assume the moduli space of solutions to the Vafa--Witten equations $\mathcal M_{VW}$ is compact and zero-dimensional (for the case $\dim\mathcal M_{VW}>0$ see example below). The index of the Hessian on $X^4\times S^1$ vanishes and therefore the relative Morse index of a pair of critical points is an integer.\footnote{In general, there is no a distinguished critical point as in the $SU(2)$-instanton Floer theory, so that we are left with the relative grading only.} The Floer differential counts the moduli space of finite-energy solutions of equations~(\ref{Eq_Main}) on $X\times\mathbb R$ converging to solutions of the Vafa--Witten equations at $\pm\infty$. As a result, for a smooth four-manifold equipped with a principal $G$-bundle Floer-type homology groups can  conjecturally be constructed.

\begin{ex}
 Let $X$ be a K\"ahler surface with a non-negative scalar curvature. Then Proposition~\ref{Prop_VWVanishingThm} applies and, therefore, $\mathcal M_{VW}=\mathcal M_{asd}$ assuming all asd connections are irreducible and non-degenerate. If $\dim\mathcal M_{asd}>0$ the function $U$ is not Morse but rather Morse--Bott. Then choosing a suitable perturbation, which is essentially a Morse function $h$ on $\mathcal M_{asd}$, one obtains the Morse--Witten complex of $h$. The details can be found for instance in~\cite{BanyagaHurtubise:09_MorseBott}. In other words, the corresponding Floer homology groups are homology groups of $\mathcal M_{asd}$. Notice that this agrees perfectly with the Vafa--Witten theory: The Vafa--Witten invariant, which counts solutions of the Vafa--Witten equations, is the Euler characteristic of $\mathcal M_{asd}$ provided the only solutions of the Vafa--Witten equations are anti-self-dual instantons.

It is worth pointing out that the above reasoning is valid if $\mathcal M_{asd}$ admits a compactification, which is a manifold. Notice that the Euler characteristic of $\mathcal M_{asd}$ in~\cite{VafaWitten:94_StrongCouplingTest} is taken as the Euler characteristic of the Gieseker compactification. 
\end{ex}

\medskip

Further, let us consider dimension three.
Let $ (Y, g) $ be a closed oriented Riemannian three-manifold.
Pick a nontrivial principal $G$-bundle $ P \rightarrow Y$ 
and assume that all solutions of Hitchin's equations are irreducible
(thus we exclude the case $ G = SU(2)$)
and the moduli space is finite, say 
$  \{ \mathscr A_1, \dots, \mathscr A_k \}$.
Recall that this is the critical set of the complex Chern--Simons functional and therefore we can conjecturally construct a corresponding collection of $k$ Fukaya--Seidel  $ A_\infty $-categories\footnote{$ \mathcal A_j(Y) $ will also depend on the metric as well as on the choice of $P$.} $\mathcal A_j(Y)$
as described in Section~\ref{Subsect_OutlineOfConstruction}.

Thus, the objects of $ \mathcal A_j(Y) $ are classes of
solutions   $\mathscr A_l$ of Hitchin's equations.  For ease of exposition we assume that $\re\CS (\mathscr A_1)<\dots <\re\CS (\mathscr A_k)$, where $\re\CS (\mathscr A_l) $ is understood to take values in $[0,1)$.   Recall that for any pair $\mathscr A_\pm \in \{ \mathscr A_1, \dots, \mathscr A_k \} $, $\mathscr A_- < \mathscr A_+ $,
the space $ hom (\mathscr A_-, \mathscr A_+) $
is generated by the broken flow lines of the complex Chern--Simons functional connecting $ \mathscr A_-$ with $\mathscr A_+$. More precisely, as described in Remark~\ref{Rem_FSCatForMultivaluedFns}, we consider only those broken flow lines $\gamma$ for which the image of $\CS\comp\gamma$ does not intersect the set %
$\bigl (\re\CS (\mathscr A_j), \re\CS (\mathscr A_{j+1})\bigr )\times\mathbb R$. Recall also that the flow lines of the complex Chern--Simons functional  can  conveniently be described as moduli of solutions of equations~\eqref{Eq_GradFlowCSwithTheta} satisfying the asymptotic conditions%
\begin{equation}\label{Eq_AsymptCondForCSFlow}
 (a,b,c,e)\longrightarrow (a_\pm^0, b_\pm^0, 0, 0)\qquad\text{as}\quad t\to\mp\infty,
\end{equation}
where $(a_\pm^0, b_\pm^0)$ are solutions of Hitchin's equations representing $\mathscr A_\pm$.

Further, the Floer differential 
$ \mu^1\colon hom (\mathscr A_-, \mathscr A_+)\rightarrow  hom (\mathscr A_-, \mathscr A_+)$ 
is obtained by counting moduli of finite-energy pseudoholomorphic planes with a Hamiltonian perturbation satisfying suitable conditions at infinity. 
In our case, by Remark~\ref{Rem_AdiabaticLimitFor5dEqns}
these pseudoholomorphic planes can (formally) be
interpreted as solutions of equations~\eqref{Eq_5dEquationsAsSVEWithParameter},
which are in turn interpreted as solutions of
equations~\eqref{Eq_Main}
on $ W = Y \times \mathbb R^2 $.

Summing up, choose any admissible pair $\mathscr B_\pm $
of gauge equivalence classes of finite--energy solutions of equations~\eqref{Eq_GradFlowCSwithTheta} and ~\eqref{Eq_AsymptCondForCSFlow}. Then  \emph{define} the map $ \mu^1 $
by counting moduli of  solutions to equations~\eqref{Eq_Main} on 
$
(W, v) = (Y \times \mathbb R^2, \cos \theta\, \partial_t + \sin \theta\, \partial_s)
$
with the following boundary conditions
\begin{align}
 ( a, b, c, e) & \rightarrow ( a_\pm(t), b_\pm(t), 0, 0 )\phantom{ a_\pm^0, b_\pm^0}  \text{as} \quad s \to \mp\infty, \nonumber\\%
 ( a, b, c, e) & \rightarrow ( a_\pm^0, b_\pm^0, 0, 0 ) \phantom{a_\pm(t), b_\pm(t)} \text{as} \quad t \to \mp \infty, \nonumber
\end{align}
where $(a_\pm(t), b_\pm(t))$ represents the class $\mathscr B_\pm$. 

To define the map $\mu^2$,
one considers finite--energy solutions of equations~\eqref{Eq_Main} on $ W = Y \times \Omega $ satisfying appropriate boundary conditions,
where $ \Omega $ is as shown in Fig.\ref{Fig_2dDomain}. The maps $ \mu^d $ for $ d \ge 3 $
are defined similarly and conjecturally
the whole collection $\{ \mu^d \}$ combines to an 
$ A_\infty $-structure.

Notice that the change of orientation on $ Y $
is equivalent to  multiplication of $ f $
by $ -1 $ and hence does not affect $\mathcal A_j(Y) $.
On the other hand, $ \mathcal A_j(Y) $ depends on the Riemannian metric $ g $. 
However, as explained in~\cite{Seidel:01VanishingCyclesI} the derived category $ D^b (\mathcal A_j(Y))$ should be independent of $g$.


\appendix
\section{Pseudoholomorphic strips and pseudoholomorphic planes}\label{Apend_PseudoholomStripsAndPlanes}

In this appendix we  outline  (without proof) a connection between pseudoholomorphic planes with a Hamiltonian perturbation and pseudoholomorphic strips with Lagrangian boundary conditions. To do so, pick a pair $(m_-, m_+)$ of critical points of $f$ and assume that the interval $\overline{z_-z_+}$ does not contain any other critical point, where $z_\pm=f(m_\pm)$. It is convenient  to choose the midpoint of  $\overline{z_-z_+}$  as the basepoint. We deviate here from our convention on the choice of the basepoint for the convenience of exposition only, namely to avoid differential equations with non-smooth coefficients.

Replacing $f$ with $e^{-i\theta_\pm}(f-z_0)$ if necessary we may assume that $z_\pm=\pm T, T>0$ and hence $z_0=0, \theta_0(t)\equiv 0$. We establish a relation between solutions of the equations
\begin{equation}\label{Eq_PseudoholomPlaneThetaVanishes}
  \begin{aligned}
    &\ \;\partial_s u + J(\partial_t u+ v_0  )=0,\qquad u\colon \mathbb R^2_{s,t}\rightarrow M,\\
    &\lim\limits_{t\to\pm\infty} u(s,t)= m_\mp,\qquad \lim\limits_{s\to\pm\infty} u(s,t)=\gamma_\mp(t)
  \end{aligned}
\end{equation}
and pseudoholomorphic strips in two steps.  In the first step we relate solutions of equations~\eqref{Eq_PseudoholomPlaneThetaVanishes} to solutions of the problem%
\begin{equation}\label{Eq_PseudoholomPlaneWithNormedHamPert}
  \begin{aligned}
    &\partial_s u_0 + J\Bigl (\partial_\tau u_0+\frac 1{\|v_0\|^2}v_0\Bigr )=0,\qquad (s,\tau)\in\mathbb R\times (-T, T),\\
    & u_0(s,\pm T)= m_\mp,\qquad \lim\limits_{s\to\pm\infty} u_0(s,\tau)=\gamma_{0,\,\pm}(\tau),
  \end{aligned}
\end{equation}
where $\gamma_{0,\,\pm}$ satisfies the equations%
\begin{equation}\label{Eq_NormedFlowLines}
  \begin{aligned}
    &\frac d{d\tau}\gamma_0 +\frac 1{\|v_0\|^2} v_0=0,\qquad \tau\in (-T, T),\\
    &\gamma_{0}(\pm T)=m_\mp.
  \end{aligned}
\end{equation}
In the second step we show how to relate solutions of~(\ref{Eq_PseudoholomPlaneWithNormedHamPert}) to pseudoholomorphic strips.

\emph{Step 1.} It is an elementary fact that equations~(\ref{Eq_NormedFlowLines}) are equivalent to the antigradient flow equations for $f_0$. Nevertheless it is instructive to examine this equivalence more closely. Consider the family of equations%
\begin{equation}\label{Eq_FlowLinesFamily}
  \begin{aligned}
      &\ \,\dot\gamma_\lambda+ \frac 1{\lambda + (1-\lambda)\| v_0\|^2}\, v_0=0,\qquad \gamma_\lambda\colon\mathbb R\rightarrow M,\\
      &\lim\limits_{t\to\pm\infty} \gamma_\lambda (t)=m_\mp,
  \end{aligned}
\end{equation}
where $\lambda\in (0, 1]$, and fix a parametrization by the condition $ f_0\comp\gamma_\lambda(0)=0$. Pick any solution $\gamma_1$ of equations~(\ref{Eq_FlowLinesFamily}) for  $\lambda=1$, i.e. an antigradient flow line of $f_0$, and consider the following family of diffeomorphisms%
\begin{equation*}
 \tau_\l\colon\mathbb R\rightarrow\mathbb R, \quad \tau_\l(t)=\l t+ (1-\l) f_0\comp\gamma_1(t),\qquad \l\in (0,1].
\end{equation*}
It is straightforward to check that $\gamma_\lambda=\gamma_1\comp\tau_\l^{-1}$ is a solution of~(\ref{Eq_FlowLinesFamily}), and this establishes a bijection between antigradient flow lines of $f_0$ and solutions of~\eqref{Eq_FlowLinesFamily}. This correspondence is also valid for  $\l=0$, but in this case $\tau_0$ maps $\mathbb R$ bijectively onto the interval $(-T, T)$. If we extend $\gamma_0$ by the constant values outside $(-T, T)$, then $\gamma_\l$ converges to $\gamma_0$ in $C^0(\mathbb R; M)$ as $\l\to 0$ (in fact, in any reasonable topology).

With this understood, consider the family of equations
\begin{equation}\label{Eq_PseudoholomPlaneWithHamPertFamily}
  \begin{aligned}
    &\ \;\partial_s u_\l + J\Bigl (\partial_t u_\l+\frac 1{\lambda + (1-\lambda)\| v_0\|^2}\, v_0\Bigr )=0,\qquad (s,t)\in\mathbb R^2\\
    &\lim\limits_{t\to\pm\infty} u_\l(s,t)= m_\mp,\qquad \lim\limits_{s\to\pm\infty} u_\l(s,t)=\gamma_{\l,\,\pm}(t).
  \end{aligned}
\end{equation}
For these equations explicit correspondence between solutions for different values of $\l$ is not available anymore, but it is reasonable to expect that $u_\l$ converges to a solution of~(\ref{Eq_PseudoholomPlaneWithNormedHamPert}) as $\l\to 0$.

\emph{Step 2.} Let $L_\pm(\tau)\subset  f^{-1}(\tau),\ \tau\in (-T,T),$ denote the vanishing cycle of $m_\pm$ associated with the segment $[\tau, \pm T]$. Consider the family of equations
\begin{equation*}
  \begin{aligned}
    &\partial_s u_\mu + J\Bigl (\partial_t u_\mu+\frac {1-\mu}{\| v_0\|^2}\, v_0\Bigr )=0,\qquad (s,\tau)\in\mathbb R\times (-T, T)\\
    & u_\mu (s,\pm T)\in L_\pm (\pm(1-\mu)T),   \qquad   \lim\limits_{s\to\pm\infty} u_\mu(s,\tau)=\gamma_{0,\,\pm}((1-\mu)\tau)
  \end{aligned}
\end{equation*}
with $\mu\in [0,1]$. Clearly, for $\mu=0$ we obtain  equations~(\ref{Eq_PseudoholomPlaneWithNormedHamPert}), whereas for $\mu=1$ we have holomorphic strips as in the classical definition of the Floer differential. Notice that the images of such holomorphic strips lie in the fiber of $f$. 

\begin{rem}
 Pick a solution $u_0$ of equations~(\ref{Eq_PseudoholomPlaneWithNormedHamPert}) and denote $ f\comp u_0=\varphi +i\psi$. It follows from the holomorphicity of $f$ that $\varphi$ and $\psi$ satisfy the inhomogeneous Cauchy-Riemann equations%
 \begin{equation*}
    \partial_s \varphi - \partial_\tau \psi=0,\quad \partial_s \psi + \partial_\tau \varphi +1=0
\end{equation*}
and therefore both functions are harmonic. Moreover, the holomorphicity of $f$ also implies that $f_1\comp\gamma_\pm (\tau)$ is constant it $\tau$ and therefore vanishes everywhere, since $f_1\comp\gamma_\pm (\pm T)= f_1(m_\pm)=0$. We conclude that $\psi$ vanishes as $\tau\to\pm T$ and as $s\to\pm\infty$ and thus vanishes everywhere. Therefore $\varphi(s,\tau)=-\tau$. We see that unlike pseudoholomorphic strips, images of solutions of~(\ref{Eq_PseudoholomPlaneWithNormedHamPert}) do not lie in a fixed fiber of $f$, but rather the fiber of $u_0(s,\tau)$ varies in a controlled manner for any $u_0$. 

Notice also that at the first glance equation~\eqref{Eq_PseudoholomPlaneWithNormedHamPert} has singularities. Namely, if a solution $u_0$ hits a critical point of $f$ at a single point $(s_0, \tau_0)$, then $\varphi$ and $\psi$ are harmonic in $\mathbb R\times (-T, T)\setminus \{ (s_0,\tau_0) \}$ and continuous at   $(s_0,\tau_0)$. Hence the singularity is removable and the above argument shows that the image of $ f\comp u_0$ is the segment $(-T, T)$. Since by assumption the segment $(-T, T)$ does not contain any critical values, we conclude that a priori a solution of~\eqref{Eq_PseudoholomPlaneWithNormedHamPert} cannot hit a critical point of $f$ in an interior point. 
\end{rem}



\section{On broken flow lines}\label{App_OnThePertAntigradFlowLines}

In this appendix missing details on broken flow lines are provided. We use notations introduced in Subsections~\ref{Subsect_OutlineOfConstruction} and~\ref{Subsect_AprioriC0Estim}.

\begin{lem}\label{Lem_DistToBrokenLine}
  Suppose the closed domain $G$ bounded by the triangle $z_-z_0z_+$   contains no critical values of $f$ other than $z_\pm$. Denote by $\ell$ the curve $\overline{z_-z_0}\cup\overline{z_0z_+}$.  Then for any $\e>0$ there exists $\nu_0>0$ such that for all broken flow lines $\gamma_\nu$ of $f$ connecting $m_-$ and $m_+$ and all $t\in\R$ we have 
\begin{equation} \label{Eq_DistanceToBrokenLine}
d\bigl (f\circ\gamma_\nu(t),\,\ell\bigr )<\e\qquad\text{provided}\quad \nu\le\nu_0.
\end{equation}
\end{lem}

\begin{proof}
  The lemma is proved in three steps.

\setcounter{thestep}{0}  

\begin{step}\label{Step_BrokenFLAreInPreimOfTriangle}
    For any broken flow line $\gamma_\nu$ the image of the curve $f\circ\gamma_\nu\colon\R\to\C$ is contained in $G$.
  \end{step}

From~\eqref{Eq_PertBrokenFlowLine} we have $\tfrac d{dt}\im f\circ\gamma_\nu (t) = -\sin\theta_\nu(t)\rho\circ\gamma_\nu\le 0$ for $t\le 0$. Since $\lim_{t\to -\infty}f\circ\gamma_\nu (t)=z_+$ we conclude that $\im f\circ\gamma_\nu(t)\le \im z_+=\zeta$ for all $t\le 0$. Similarly, $\im f\circ\gamma_\nu(t)\le\im z_-=\zeta$ for all $t\ge 0$. Hence, the image of the curve $f\circ\gamma_\nu$ lies in the half--plane, which is bounded by the straight line through $z_-$ and $z_+$  and contains $z_0$. Arguing along similar lines, one also obtains that the image of $f\circ\gamma_\nu$ is contained in the half--plane bounded by the straight line through $z_\pm$ and $z_0$ and containing $z_\mp$.

  \begin{step}\label{Step_BehOfBrokenFlowLineAtInfty}
For any $\e>0$ there exist $T_\e>0$ and $\nu_0=\nu_0(\e)>0$ such that $f\comp\gamma_\nu(\pm t)\in B_\e(z_\mp)$ for all $t\ge T_\e$ and all $\nu\le\nu_0$.
  \end{step}

We prove that  $f\comp\gamma_\nu( t)\in B_\e(z_-)$ for all $t\ge T_\e$ and all $\nu\le\nu_0$. The rest can be proved similarly.

 For an arbitrary $\e>0$ denote
\begin{align*}
  \rho_\e &=\inf\bigl\{\rho(m)\mid f(m)\in G,\  \im f(m)\le\zeta-\e\bigr \}>0,\\ 
T_\e &=1+\frac {1+\zeta}{\rho_\e\sin\theta_-},\qquad \nu_0= T_\e^{-1}.
\end{align*}
We claim that $\im f\comp\gamma_\nu(T_\e)> \zeta-\e$ for any $\nu\le \nu_0$. Indeed, assume this is not the case, i.e. there exists $\nu\le\nu_0$ such that $\im f\comp\gamma_\nu(T_\e)\le \zeta-\e$.  Then for any $t\in [1,T_\e]$ we have $\im f\comp\gamma_\nu(t)\le \zeta-\e$ since the function $\im f\comp\gamma_\nu(t)$ is monotone for $t\ge 0$ as indicated in the proof of Step~\ref{Step_BrokenFLAreInPreimOfTriangle}. Hence, 
\begin{equation*}
\zeta\ge \im{ \bigl ( f\comp\gamma_\nu(T_\e) - f\comp\gamma_\nu(1)\bigr )}=
 \sin\theta_- \int_1^{T_\e} \rho\circ\gamma_\nu (t)\, dt\ge
  \sin\theta_- \rho_\e (T_\e-1)\ge 1+\zeta.
\end{equation*} 
This contradiction proves the inequality $\im f\comp\gamma_\nu(T_\e)> \zeta-\e$, which in turn implies that $\im f\comp\gamma_\nu(t)> \zeta-\e$ for all $t\ge T_\e$ and all $\nu\le\nu_0$. Arguing along similar lines and redenoting $T_\e, \nu_0$ if necessary one also obtains that the inequality $\im{\bigl ( e^{-i\theta_+}f\comp\gamma_\nu(t)  \bigr )}\ge \re (e^{-i\theta_+}z_-)-\e$ holds for all $t\ge T_\e$ and all $\nu\le\nu_0$. This implies Step~\ref{Step_BehOfBrokenFlowLineAtInfty}.



  \begin{step}
    We prove  the lemma.
  \end{step}

Pick any $\e>0$. Then by Step~\ref{Step_BehOfBrokenFlowLineAtInfty} there exist $T_\e>0$ and  $\nu_0\le T_\e^{-1}$ such that $\bigr |f\circ\gamma_\nu(t) - z_- \bigl |<\e$ holds for all $t\ge T_\e$ provided $\nu\le \nu_0$. Since for $t\in [\nu, \nu^{-1}]$ we have that $f\circ\gamma_\nu(t)$ lies on a straight line parallel to the straight line through $z_0$ and $z_-$, we obtain that inequality~\eqref{Eq_DistanceToBrokenLine}  
holds for all $t\ge\nu$. Using similar arguments one shows that inequality~\eqref{Eq_DistanceToBrokenLine} also holds for $t\le -\nu$. Furthermore, the length of the curve $f\circ\gamma_\nu(t),\ t\in [-\nu,\nu],$ is bounded by  $2\sqrt{\bar\rho}\nu$, where $\bar\rho=\sup \{\rho(m)\mid m\in M\}$.
 Redenoting $\nu_0$ if necessary we obtain that inequality~\eqref{Eq_DistanceToBrokenLine} holds for $t\in [-\nu,\nu]$ as well.
\end{proof}


\begin{proof}[Proof of Proposition~\ref{Prop_ConvergentSubseqOfBrokenFlowLines}]

The proof consists of the following four steps

\setcounter{thestep}{0}

\begin{step}\label{Step_1Prop_ConvergentSubseq}
  For any $\e>0$ there exists $\nu_0>0$ such that $d ( \gamma_\nu(t), m_-  )<\e$ for all $\nu\le\nu_0,\ t\ge \nu^{-1}+1$, and $\gamma_\nu\in\Gamma_\nu(m_-, m_+)$.
\end{step}

From the equality $\tfrac d{dt}\re (f\comp\gamma_\nu(t)) = -|\dot\gamma_\nu(t)|^2$, which is valid for all  $t\ge \nu^{-1}+1$, we obtain
\begin{equation}\label{Eq_AuxW12NormOfBFL}
  \int_{\nu^{-1}+1}^\infty |\dot\gamma_\nu(t)|^2\, dt= \re(f\comp\gamma_\nu(\nu^{-1}+1))-\re z_-
\end{equation}
Hence, the map $\beta_\nu(t)=\gamma_\nu(\nu^{-1}+1+t)$ belongs to $W^{1,2}(\R_+; M)$. Moreover, by Step~\ref{Step_BehOfBrokenFlowLineAtInfty} in the proof of Lemma~\ref{Lem_DistToBrokenLine} there exists $\nu_0=\nu_0(\e)$ such that $\| \beta_\nu \|_{W^{1,2}}<\e$ for all $\nu\le\nu_0$. Hence, Step~\ref{Step_1Prop_ConvergentSubseq} follows from the Sobolev embedding $W^{1,2}(\R_+;M)\hookrightarrow C^0(\R_+; M)$.

\begin{step}\label{Step_2Prop_ConvergentSubseq}
  For any $\e>0$ there exists $\nu_0=\nu_0(\e)>0$ such that $|\dot\gamma_\nu(t)|<\e$ for all $\nu\le\nu_0,\ t\in [\nu^{-1},\nu^{-1}+1]$, and $\gamma_\nu\in\Gamma_\nu(m_-, m_+)$.
\end{step}

By choosing local coordinates we can identify a neighbourhood of $m_-$ with $\R^{2n}$. Since $m_-$ is a nondegenerate critical point of $f_0=\re f$, we can assume that $f_0$ is a quadratic function in the local representation. Hence there exist positive constants $C$ and $\delta$ such that the inequality $\rho(x)\le C^2|x|^2$ holds whenever $|x|\le \delta$. Here $|\cdot |$ is the standard Euclidean norm on $\R^{2n}$.  Hence,
\begin{equation}\label{Eq_AuxLocEstimForBFL}
|\dot\gamma_\nu|^2 =\rho(\gamma_\nu)\le C^2|\gamma_\nu|^2
\end{equation}
provided $|\gamma_\nu|\le\delta$. Therefore, for any $\tau\ge 0$ we have
\begin{equation}\label{Eq_AuxLocEstimForBFL2}
  \begin{aligned}
  |\gamma_\nu(\nu^{-1}+\tau)| &= 
 \bigl | \gamma_\nu(\nu^{-1}+1) -\int_{\nu^{-1}+\tau}^{\nu^{-1}+1}\dot\gamma_\nu(t)\, dt  \bigr |\\
 &\le |\gamma_\nu(\nu^{-1}+1)|+ C\int_{\nu^{-1}+\tau}^{\nu^{-1}+1}|\gamma_\nu(t)|\, dt    
  \end{aligned}
 \end{equation}
provided $|\gamma_\nu(t)|<\delta$ for all $t\in [\nu^{-1}+\tau, \nu^{-1}+1]$. 

For any $\e>0$ such that $\e/C<\delta$ by Step~\ref{Step_1Prop_ConvergentSubseq} we can choose $\nu_0>0$ so small that the inequality
\[
|\gamma_\nu(\nu^{-1}+1)|< \frac {e^{-C}\e}C <\delta
\]
holds for all $\nu\le\nu_0$ and $\gamma_\nu\in\Gamma_\nu(m_-, m_+)$. In particular,~\eqref{Eq_AuxLocEstimForBFL2} holds for all $\tau$ sufficiently close to $1$. By the Gronwall--Bellman inequality we obtain
\begin{equation}\label{Eq_AuxLocEstimForBFL3}
|\gamma_\nu(\nu^{-1}+\tau)|\le|\gamma_\nu(\nu^{-1}+\tau)| e^{C(1-\tau)}<\e/C
\end{equation}
for all $\tau\in [0,1]$ sufficiently close to $1$. This implies in fact that~\eqref{Eq_AuxLocEstimForBFL3} holds for all $\tau\in [0,1]$. Then Step~\ref{Step_2Prop_ConvergentSubseq} follows by~\eqref{Eq_AuxLocEstimForBFL}.

\begin{step}\label{Step_3Prop_ConvergentSubseq}
For any $\e>0$ there exists $T_\e>0$ and $\nu_0\le T_\e^{-1}$ such that $d(\gamma_\nu(t), m_-)<\e$ for all $\nu\le\nu_0,\ t\ge T_\e$, and $\gamma_\nu\in\Gamma_\nu(m_-, m_+)$.  
\end{step}

Let $T_\e$ be as in Step~\ref{Step_BehOfBrokenFlowLineAtInfty} in the proof of Lemma~\ref{Lem_DistToBrokenLine}. From the equality  $\tfrac d{dt}\re{\bigl ( e^{-i\theta_-}f\comp\gamma_\nu(t)\bigr )} = |\dot\gamma_\nu(t)|^2$, which is valid for all  $t\in [ \nu, \nu^{-1}]$, we obtain 
\begin{equation*}
  \int_{T_\e}^{\nu^{-1}} |\dot\gamma_\nu(t)|^2\, dt
=\Bigl | \re{\bigl (e^{i\theta_-}f\comp\gamma_\nu(\nu^{-1})\bigr )}- 
  \re{\bigl (e^{i\theta_-}f\comp\gamma_\nu(T_\e)\bigr )}\Bigr |
\le 2\e
\end{equation*}
By~\eqref{Eq_AuxW12NormOfBFL} and Step~\ref{Step_2Prop_ConvergentSubseq} we obtain
\[
\begin{aligned}
  \int_{T_\e}^\infty|\dot\gamma_\nu|^2\, dt 
&=\int_{T_\e}^{\nu^{-1}}|\dot\gamma_\nu|^2\, dt 
+ \int_{\nu^{-1}}^{\nu^{-1}+1}|\dot\gamma_\nu|^2\, dt
+ \int_{\nu^{-1}+1}^{\infty}|\dot\gamma_\nu|^2\, dt\\
&\le 2\e +\e^2 +\e.
\end{aligned}
\]
Step~\ref{Step_3Prop_ConvergentSubseq} follows from the embedding $W^{1,2}(\R_+;M)\hookrightarrow C^0(\R_+; M)$.

\begin{step}
  We prove the proposition.
\end{step}

Since $\rho$ is bounded on $M$, we obtain that  $|\dot\gamma_j|^2\le \sup\rho <\infty$. Then with the help of the Ascoli-Arzela theorem we can find a subsequence $\gamma_{j_k}$, which converges to some $\gamma_0\in C^0(\mathbb R; M)$ on each compact interval. Then $\gamma_0\in C^1(\mathbb R\setminus\{0\}; M)$ and satisfies %
$\dot\gamma_0 +\cos\theta_0\, v_0+\sin\theta_0\, v_1=0$.  By Lemma~\ref{Lem_DistToBrokenLine} the image of  $f\comp\gamma_0$ is contained in $\ell=z_-z_0z_+$. Hence, the limit $\lim_{t\to \pm\infty}\gamma_0(t)$ must be a critical point of the vector field $\cos\theta_\mp v_0 +\sin\theta_\mp v_1$. Hence, $\lim_{t\to \pm\infty}\gamma_0(t)=m_\mp$, i.e. $\gamma_0$ is a solution of~\eqref{Eq_BrokenFlowLine}.

Choose any $\e>0$ sufficiently small. By Step~\ref{Step_3Prop_ConvergentSubseq} there  exist $T_\e>0$ and $N_\e>0$ such that  
\begin{equation}\label{Eq_AuxDistanceBetwBFLAndGamma0}
d(\gamma_{j_k}(t),\gamma_0(t))<\e
\end{equation}
provided $t\ge T_\e$ and $k\ge N_\e$. Arguing similarly, we also obtain that~\eqref{Eq_AuxDistanceBetwBFLAndGamma0}  holds for $t\le -T_\e$ (possibly after increasing $T_\e$). Since $\gamma_{j_k}$ converges on $[-T_\e, T_\e]$ we can find $N_\e'\ge N_\e$ such that~\eqref{Eq_AuxDistanceBetwBFLAndGamma0} also holds for $t\in [-T_\e, T_\e]$ provided $j_k\ge N_\e'$. This finishes the proof.
\end{proof}

Introduce the Banach manifold %
\begin{equation*}\label{Eq_AuxW12RM}
W^{1,2}_{m_-\!;\, m_+}=\{\, \gamma\in W^{1,2}(\mathbb R; M)\; |\; \lim\limits_{t\to\pm\infty}\gamma(t)=m_\mp  \}
\end{equation*}
and the vector bundle $\mathcal E\rightarrow W^{1,2}_{m_-\!;\, m_+}$, whose fiber at $\gamma$ is the Hilbert space $L^2(\gamma^*TM)$. The map $\sigma_\nu$ given by~\eqref{Eq_Map_sigma} can be interpreted as a section of $\mathcal E$. Similarly, the map~\eqref{Eq_DsigmaW12L2} can be interpreted as the covariant derivative of $\sigma_\nu$.  Then $\sigma_\nu$  is a Fredholm section~\cite{Salamon:90MorseTheory} with vanishing index, since the Morse indices of $m_+$ and $m_-$ are equal. Here $m_\pm$  is regarded as a critical point of $\re f$. Clearly, $\sigma_\nu^{-1}(0)=\Gamma_\nu(m_-;m_+)$.

\begin{lem}\label{Lem_TransversalityForBrokenFlowLines}
Let $L_\pm$ be the vanishing cycle corresponding to the segment $\overline{z_\pm z_0}$. If  $L_+$ and $L_-$ intersect transversely in $ f^{-1}(0)$, then there exists $\nu_0>0$ such that $\sigma_\nu$ intersects the zero section transversely for all $\nu\in (0,\nu_0)$. Moreover there exists a natural bijective correspondence between %
  $\Gamma_\nu(m_-;m_+)$ and $\Gamma_0(m_-;m_+)$ provided  $\nu\in (0,\nu_0)$.
\end{lem}


\begin{proof}
 Let $\mathcal U_+$ denote the unstable manifold of $m_+$ regarded as a critical point of $\re (e^{-i\theta_+}f)$. Similarly, let $\mathcal S_-$ denote the stable manifold of $m_-$ regarded as a critical point of $\re (e^{-i\theta_-}f)$.

Pick a point $m\in L_-\cap L_+\cong \mathcal S_-\cap\mathcal U_+$ and observe that $\mathcal S_-$ and $\mathcal U_+$ are the Lagrangian thimbles of $m_-$ and $m_+$ associated with the segments $\overline{z_0z_-}$ and $\overline{z_0z_+}$, respectively. Here $z_0$ is the origin. Then the hypothesis of the lemma implies that $\mathcal S_-$ and $\mathcal U_+$ intersect transversally at $m$.

Let $\gamma_0$ be the solution of~\eqref{Eq_BrokenFlowLine} corresponding to $m$. Denote by $D_{\gamma_0}\sigma_0$  the linearization of $\sigma_0$ at $\gamma_0$. As we have already remarked above,  $D_{\gamma_0}\sigma_0\colon W^{1,2}(\gamma_0^*TM)\rightarrow L^2(\gamma_0^*TM)$ is a Fredholm operator of index $0$. Moreover, it can be shown in the similar manner as in the proof of Theorem~3.3 in~\cite{Salamon:90MorseTheory} that $\dim\coker D_{\gamma_0}\sigma_0=\codim (T_m\mathcal S_-+T_m\mathcal U_+)$. Therefore $\dim\coker D_{\gamma_0}\sigma_0=0$ and hence $\dim\ker D_{\gamma_0}\sigma_0=0$. Thus we conclude that $\sigma_0$ intersects the zero-section transversely.

It follows from Proposition~\ref{Prop_ConvergentSubseqOfBrokenFlowLines} that there exists $\nu_0>0$ such that each solution of the equation $\sigma_\nu(\gamma_\nu)=0,\ \nu\in (0, \nu_0)$ is contained in a $C^0$--neighbourhood  $U(\gamma_0)$ of some $\gamma_0\in\sigma_0^{-1}(0)$. Notice that the linearization of $\sigma_\nu$ at $\gamma_\nu$ can be written in the form%
\begin{equation*}
 \begin{aligned}
  D_{\gamma_\nu}\sigma_\nu (\xi) &=\nabla_{\dot\gamma_\nu} \xi +\cos\theta_\nu\nabla_\xi v_0 +\sin\theta_\nu\nabla_\xi v_1\\%
  &=\cos\theta_\nu \bigl (\nabla_{\xi} v_0  - \nabla_{v_0} \xi\bigr ) +\sin\theta_\nu\bigl (\nabla_{\xi} v_1  - \nabla_{v_1} \xi\bigr ).
 \end{aligned}
\end{equation*}
where $\xi\in W^{1,2}(\gamma_\nu^*TM)$. Hence, redenoting $\nu_0$ if necessary,  we can assume that the linearization of $\sigma_\nu$ is non-degenerate at each $\gamma_\nu\in\Gamma_\nu(m_-; m_+)$ contained in $\bigcup_{\gamma_0}U(\gamma_0)$ for $\nu\in (0, \nu_0)$, since $\#\sigma_0^{-1}(0)=\# L_-\cap L_+<\infty$. Thus $\sigma_\nu$ intersects the zero-section transversely provided $\nu\le\nu_0$.

\medskip

Consider $\sigma_\nu$ as a section of $\pi^*\mathcal E$, where $\pi\colon W^{1,2}_{m_-\!;\, m_+}\times\mathbb R_\nu\rightarrow \mathbb R_\nu$ is the canonical projection. Then $\sigma_\nu$ is continuous and satisfies the hypothesis of the implicit function theorem. Therefore, %
 $\{ (\gamma,\nu)\; |\; \sigma_\nu(\gamma)=0,\ \nu\in [0, \nu_0)\}$ is homeomorphic to $\sigma_0^{-1}(0)\times  [0, \nu_0)$. This establishes the bijective correspondence between $\Gamma_0(m_-; m_+)$ and $\Gamma_\nu(m_-; m_+)$. 
\end{proof}

\begin{cor}
  If  $L_+$ and $L_-$ intersect transversely in $ f^{-1}(0)$, then there exists $\nu_0>0$ such that hypothesis~(H\ref{Hyp_NondegOfBrokenFlowLines}) holds provided $\nu\le\nu_0$.
\end{cor}

\end{document}